\documentclass[a4paper,12pt,twoside]{amsart}
\usepackage[left=3cm,right=3cm,top=4.5cm,bottom=4cm]{geometry}
\usepackage[english]{babel}
\usepackage[latin1]{inputenc}
\usepackage{amsmath}
\usepackage{amsthm}
\usepackage{amscd}
\usepackage[all]{xy}
\usepackage{amssymb}
\usepackage{amsfonts}
\usepackage{mathrsfs}
\usepackage{enumerate}
\usepackage{comment}
\usepackage[usenames, dvipsnames]{color}
\allowdisplaybreaks

\usepackage{hyperref}
\hypersetup{colorlinks,breaklinks,
           linkcolor=MidnightBlue,urlcolor=MidnightBlue,
                       anchorcolor=MidnightBlue,citecolor=MidnightBlue}

\newtheorem{thm}{Theorem}[section]
\newtheorem{prop}[thm]{Proposition}
\newtheorem{cor}[thm]{Corollary}
\newtheorem{lem}[thm]{Lemma}
\theoremstyle{definition}
\newtheorem{dfn}[thm]{Definition}
\newtheorem{ex}[thm]{Example}

\newtheorem{rem}[thm]{Remark}
\newtheorem{rems}[thm]{Remarks}

\usepackage{amsfonts}
\newcommand{\ds}{\displaystyle}
\newcommand{\field}[1]{\mathbb{#1}}
\newcommand{\Q}{\field{Q}}
\newcommand{\C}{\field{C}}

\newcommand{\N}{\field{N}}
\newcommand{\Z}{\field{Z}}

\newcommand{\F}{\field{F}}
\newcommand{\G}{\Gamma}
\renewcommand{\L}{\Lambda}

\DeclareMathOperator{\Spec}{Spec}

\DeclareMathOperator{\rank}{rank}
\DeclareMathOperator{\car}{char}

\DeclareMathOperator{\Gal}{Gal}

\DeclareMathOperator{\Hom}{Hom}

\DeclareMathOperator{\Stab}{Stab}

\DeclareMathOperator{\Ker}{Ker}

\DeclareMathOperator{\Sel}{Sel}

\DeclareMathOperator{\Fr}{Fr}

\newcommand{\im}{\text{Im}}

\newcommand{\iri}{\hookrightarrow}

\newcommand{\liminv}{\varprojlim}

\newcommand{\calf}{\mathcal F}

\newcommand{\cals}{\mathcal S}

\newcommand{\g}{\gamma}

\newcommand{\gotm}{\mathfrak m}

\newcommand{\gotp}{\mathfrak p}
\newcommand{\gotq}{\mathfrak q}

\renewcommand{\aa}{\mathbf{a}}
\newcommand{\bb}{\mathbf{b}}
\newcommand{\pp}{\mathbf{p}}
\newcommand{\bmu}{\boldsymbol\mu}

\newcommand{\ov}{\overline}

\newfont{\cyr}{wncyr10 scaled 1000}
\newcommand{\sha}{\mbox{{\cyr{X}}}}

\begin{document}
\title[{The algebra $\Z_\ell[[\Z_{\lowercase{p}}^{\lowercase{d}}]]$} and Iwasawa theory]{The algebra $\Z_\ell[[\Z_p^d]]$ and applications to Iwasawa theory}

\author[A. Bandini] {Andrea Bandini}
\address{Andrea Bandini \newline
\indent Dipartimento di Matematica, Universit\`a degli Studi di Pisa \newline
\indent Largo Bruno Pontecorvo 5 \newline
\indent 56127 Pisa, Italy}
\email{andrea.bandini@unipi.it}

\author[I. Longhi]{Ignazio Longhi}
\address{Ignazio Longhi \newline
\indent Dipartimento di Matematica, Universit\`a di Torino \newline
\indent via Carlo Alberto 10 \newline
\indent 10123 Torino, Italy}
\email{ignazio.longhi@unito.it}

\begin{abstract} Let $\ell$ and $p$ be distinct primes, and let $\G$ be an abelian 
pro-$p$-group. We study the structure of the algebra $\L:=\Z_\ell[[\G]]$ and of 
$\L$-modules. The algebra $\L$ turns out to be a direct product of copies of
ring of integers of cyclotomic extensions of $\Q_\ell$ and this induces a similar 
decomposition for a family of $\L$-modules.  Inside this family we
define Sinnott modules and provide characteristic ideals and formulas \`a la Iwasawa 
for orders and ranks of their quotients. 
When $\G\simeq \Z_p^d$\, is the Galois group of an 
extension of global fields, $\ell$-class groups and (duals of) $\ell$-Selmer groups 
provide examples of Sinnott modules and our formulas vastly extend
results of L.~Washington and W.~Sinnott on $\ell$-class groups in 
$\Z_p$-extensions. 

Moreover, for global function fields of positive characteristic we
use the specialization of a Stickelberger series to define an element 
in $\L$ which interpolates special values of Artin $L$-functions. 
With this element and the characteristic ideal of  $\ell$-class groups we formulate an Iwasawa Main Conjecture for this setting and prove 
some special cases of it for relevant $\Z_p$-extensions. 
\end{abstract}



\maketitle

\section{Introduction}
Fix a prime $p$ and let $k$ be a global field. Let $K/k$ be a 
Galois extension with Galois group $G\simeq \Z_p$, and let $\Z_p[[G]]$ be the 
associated Iwasawa algebra, which is (non-canonically) isomorphic to the ring of power 
series $\Z_p[[t]]$. 
Classical Iwasawa theory studies $p$-parts of $G$-modules (hence $\Z_p[[G]]$-modules) associated with the 
finite subextensions $k_n$ of $K$. The first examples are class groups and Selmer groups, but the theory and main results have been vastly generalized
to other arithmetically significant modules and to general $p$-adic Lie extensions
(for a comprehensive survey in the classical setting with a rich 
bibliography see \cite{GrI}). 

\subsection{Classical Iwasawa theory}
One of the main features of Iwasawa theory is its interplay between algebraic and 
analytic methods, leading to various formulations of the Main Conjecture. As an example 
we can consider $p$-parts of class groups in a $\Z_p$-extension. Let $X_p(k_n)$ be the 
$p$-part of the class group of $k_n$, and let $X_p(K):= \displaystyle{\liminv 
X_p(k_n)}$ be the inverse limit with respect to the norm maps, which we refer to as the
$p$-part of the class group of $K$. Then, $X_p(K)$ is a finitely generated torsion module
over $\Z_p[[G]]\simeq \Z_p[[t]]$ and, using the structure theorem for such modules, it 
is possible to associate with $X_p(K)$ a polynomial $f_X(K,t)\in \Z_p[[t]]$
(a generator of the so called {\em characteristic ideal} of $X_p(K)$).
Let $\mu$ be the maximal power of $p$ dividing $f_X(K,t)$, and let
 $\lambda=\deg_t f_X(K,t)$. Then, we have the celebrated Iwasawa formula 
\begin{equation} \label{e:iwswfrm} 
v_p (|X_p(k_n)|)= \mu p^n +\lambda n + \nu \quad \text{ for all } n\gg 0,\end{equation} 
where $v_p$ is the usual normalized $p$-adic valuation with $v_p(p)=1$.

On the analytic side, one can find the same arithmetic information using the class number 
formula, i.e., computing special values of complex Dedekind $\zeta$-functions or of products 
of complex Artin $L$-functions. Interpolating several of those special values with respect to 
a $p$-adic metric, Iwasawa provided a $p$-adic $L$-function $L_p(K,s)$ ($s$ a 
$p$-adic variable), which can also be viewed as an element $f_p(K,t)$ in $\Z_p[[t]]$.

Therefore, we have two elements in the Iwasawa algebra $\Z_p[[t]]$, which are 
constructed with completely unrelated techniques, and, nevertheless, provide the same 
arithmetic information, that is, the cardinality of all the $X_p(k_n)$. Roughly speaking, the Iwasawa Main Conjecture predicts 
explicit relations between (the $\chi$-parts of) the two ideals $(f_X(K,t))$ and $(f_p(K,t))$
in certain specific $\Z_p$-extensions like the cyclotomic ones (the $\chi$ are characters associated with the arithmetic setting).

\subsection{The $\ell$-part of Iwasawa modules}
Regarding the $\ell$-parts of class groups in $\Z_p$-extensions, with $\ell$ a prime different from $p$, for quite some time
results have been basically limited to the two foundational papers \cite{Wa75} and 
\cite{Wa78} by L.~Washington, in which he proved (among other things):
\begin{itemize}
\item bounds on the orders and $\ell$-ranks\footnote{We recall that 
the $\ell$-rank of a finitely generated $\Z_\ell$-module $M$ is 
$\rank_\ell M:=\dim_{\F_\ell}M/\ell M$.} of $\ell$-parts of class groups,
\cite[Theorem 1 and Proposition 1]{Wa75};
\item the existence of (non-cyclotomic) $\Z_p$-extensions with unbounded 
$\ell$-parts of 
class groups, \cite[Section VI]{Wa75};
\item the finiteness of the $\ell$-part of the inverse limit of 
class groups in the $\Z_p$-cyclotomic extension of an abelian number 
field, \cite[Theorem]{Wa78}.
\end{itemize}
Moreover, in a talk at the Iwasawa Conference (MSRI 1985), W.~Sinnott
presented some statements on the existence of the $p$-adic limit of 
the orders of $\ell$-class groups in a $\Z_p$-extension. A proof and some 
generalizations of Sinnott's Theorem are provided in \cite{Ki}.
\medskip

This topic has recently attracted new attention with the works of 
D.~Kundu and A.~Lei on bounds for the $\ell$-part of Selmer groups 
in anticyclotomic $\Z_p$-extensions (\cite{KL}), of A.~Lei and
D.~Valli\`eres on Iwasawa theory for graphs (\cite{LV}), 
and with a new proof by M.~Ozaki of Sinnott's Theorem, which extends it 
to the non-abelian case (\cite{Oz}). Basically, most of these 
results rely on the analytic side of Iwasawa theory, and the available 
proofs use class number formulas, complex and $p$-adic $L$-functions, 
$p$-adic interpolation, and so on. 
Because of this, some of them are limited to specific settings in which 
$L$-functions are already known to exist 
(cyclotomic and anticyclotomic $\Z_p$-extensions mostly).
\medskip

L.~Washington also remarked that to study $\ell$-parts of Iwasawa 
modules ``we would have to work with $\Z_\ell[[\Z_p]]$, which does not 
have as nice a structure'' compared to $\Z_p[[\Z_p]]\simeq \Z_p[[t]]$  
(last line on \cite[page 87]{Wa78}).
At present, there does not seem to be a well developed algebraic theory 
for this setting:
we aim at filling this gap by providing a description of the Iwasawa algebra 
$\Lambda:=\Z_\ell[[\G]]$ for any abelian pro-$p$-group $\G$, and of certain
$\L$-modules.

\subsection{$\Lambda$-modules}
Let $\G^\vee$ be the Pontrjagin dual of $\G$. The natural action of 
the absolute Galois group of $\Q_\ell$ on $\G^\vee$ yields a set of orbits 
$\G^\vee/G_{\Q_\ell}$ that we denote by $\mathscr{R}$. 
 In Section
\ref{s:IwaAlg} we shall prove the following structure theorem.

\begin{thm}[Theorem \ref{t:struttura1}]\label{t:Intro1}
\[ \L\simeq \prod_{[\omega]\in \mathscr{R}} \Z_\ell[\omega(\G)]
\simeq \prod_{n\geqslant 0} 
\prod_{\begin{subarray}{c} [\omega]\in \mathscr{R}\\
\omega(\G)=\boldsymbol{\mu}_{p^n}\end{subarray}} 
\Z_\ell[\boldsymbol{\mu}_{p^n}].\]
\end{thm}

The proof is based on a rather simple idea, namely using the fact that 
$p$ is a unit in $\Z_\ell$ to construct idempotents of $\L$ yielding 
this decomposition. \medskip

The structure theory of $\L$-modules is a straightforward consequence 
of Theorem \ref{t:Intro1}.
In particular, one can define the characteristic ideal of a locally 
finitely generated $\L$-module (Section \ref{SecCharId}). 

\noindent Moreover, under some natural conditions 
(Section \ref{SecSinMod2}), we obtain a decomposition 
\[ M\simeq \prod_{n\geqslant 0} 
\bigoplus_{\begin{subarray}{c} [\omega]\in \mathscr{R} \\
\omega(\G)=\boldsymbol{\mu}_{p^n}\end{subarray}} M_{[\omega]},\]
where all $M_{[\omega]}$ are $\Z_\ell[\boldsymbol{\mu}_{p^n}]$-modules.
This immediately leads to rather explicit formulas for $\Z_\ell$-ranks,
$\ell$-ranks and (whenever feasible) orders of quotients of locally
finitely generated or locally torsion $I_\bullet$-complete $\L$-modules
(Theorem \ref{t:AdmMod1&2}). 

We call these modules {\em Sinnott modules}, because they verify 
a generalization of Sinnott's Theorem mentioned above (Theorem 
\ref{t:GenSinn}). In Section \ref{s:NorSys} we show that a natural 
way to obtain $I_\bullet$-complete modules, and hence Sinnott modules, 
is via normic systems.
\medskip

The applications we offer are to $\G=\Z_p^d$-extensions $\calf/F$ of a global 
field $F$, taking as Iwasawa modules either $\ell$-class groups or 
$\ell$-Selmer groups of an abelian variety. We obtain
formulas for $\ell$-ranks and orders of $\ell$-class groups $X_n$ and for 
$\Z_\ell$-ranks and $\ell$-ranks 
of Pontrjagin duals $\mathcal{S}_n$ of $\ell$-Selmer groups associated 
with the finite subextensions $F_n=\calf^{\G^{p^n}}$  
of $\calf$ (see Section \ref{s:ArAppl} for the precise definitions of 
$X_n$ and $\mathcal{S}_n$ in the number field and function field setting).

\begin{thm}\label{t:Intro2}
Let $\ell$ and $p$ be distinct primes, and let $\calf/F$ be a $\Z_p^d$-extension of global fields. Let $Z_n$ be either $X_n$ or $\mathcal{S}_n$ as above, and 
let $Z=\varprojlim Z_n$.
Then:\begin{enumerate}[{\rm 1.}]
\item $Z$ is a Sinnott module.
\item Let $\Phi_\ell(Z_n)$ be $v_\ell(|Z_n|)$ (when $Z_n$ is finite), or 
$\rank_{\Z_\ell} Z_n$ or $\rank_\ell Z_n$
(when $Z_n$ is finitely generated over $\Z_\ell$).
Let also $f_m:=[\Q_\ell(\boldsymbol{\mu}_{p^m}):\Q_\ell]$.
Then, for all $m \geqslant 0$, there exist non-negative integers 
$c_m=c_m(Z)$ such that 
\[ \Phi_\ell (Z_n)  = c_0 + \sum_{m=1}^n c_mf_m \quad
\text{for all }n\geqslant 0\,.\]
\item The sequences
\[ \{ \rank_{\Z_\ell} Z_n\}_{n\in\N},\ \  \{ \rank_\ell Z_n\}_{n\in\N}\ 
\text{ and }\ \left\{ \ell^{v_\ell(|Z_n|)}\right\}_{n\in \N} \]
(when well defined) converge p-adically. 
\end{enumerate}
\end{thm}

These formulas will directly follow from the ones in Theorem
\ref{t:AdmMod1&2} (with the adjustment explained in Remark \ref{r:RnSn}),
which provide an alternative and more general approach to some of the 
results mentioned before (Section \ref{SecApp1}).  

\begin{rem} In order to have results on the behaviour of class groups or
Selmer groups in Iwasawa towers, usually one inserts some restrictions on 
the ramification locus of $\calf/F$, which is usually  assumed to 
be finite (a non-trivial hypothesis if $\car(F)=p$) and, in the case of 
the Selmer group of an abelian variety, one also asks for some 
conditions on the reduction at ramified places. Thus it is 
worth observing that our theorems are independent on any such 
restrictions. On the other hand, they will resurface in the analytic 
side of the theory and its arithmetic applications (Section \ref{s:StickEl}). 
\end{rem}

\subsection{Looking for a Main Conjecture}
Having developed an algebraic theory of $\L$-modules, the next natural step is to define $\ell$-adic
$L$-functions in $\L$ and compare it with generators of relevant characteristic ideals.

In Section \ref{s:StickEl} we focus on the concrete example of $\ell$-class groups for 
$\Z_p^d$-extensions $\calf/F$ of global function fields, with $d\in \N$.
Our approach is based on the Stickelberger series $\Theta_{\calf/F,S}(u)$ defined in \cite[Section 3.1]{ABBL}, where $S$ is a non-empty finite set of places containing the ramified ones.
This series interpolates the values of Artin $L$-functions and provides 
the $p$-adic $L$-function used in \cite{ABBL}  to prove the Iwasawa Main Conjecture
for the $p$-part of the class group of a Carlitz cyclotomic extension of $\F_q(t)$.

Special values of Artin $L$-functions give the order of the whole class group, hence the above
interpolation property can be used for $\ell$-parts as well. We use a (slightly) modified version $\Theta_{\calf/F,S,v_0}(u)$ ($v_0$
a place outside $S$) of the Stickelberger series %of \cite{ABBL} 
to define an element $\theta_{\calf/F,S,v_0}\in\Lambda$. However, the structure of $\Lambda$ (Theorem \ref{t:Intro1}) forces us to take as 
its $[\omega]$-component a product of $\chi$-factors 
$\chi(\Theta_{\calf/F,S,v_0})(1)$, as $\chi$ varies in $[\omega]$ 
(Definition \ref{d:ModStickEl}; for more details on the motivation 
behind this technical step see Section \ref{ss:ExtConstMC}).
 
With this Stickelberger element and some technical (but not restrictive) hypotheses,
we are able to prove a weak version of the Main Conjecture (Theorem \ref{t:WMC})
and some instances of the Main Conjecture for $\Z_p$-extensions (Theorem \ref{t:MCArith}).
As an example, in the case of the arithmetic 
$\Z_p$-extension $F^{ar}/F$ we obtain the following.

\begin{thm}[Corollary \ref{co:IMCZpArith}]\label{t:IntroMC}
Let $X$ be the inverse limit (with respect to norms) of the $\ell$-parts of class groups in the arithmetic tower $F^{ar}/F$.
Assume $\ell$ is a generator of $(\Z/p\Z)^*$ and $v_p(\ell^{p-1}-1)=1$.
Then there exist $\ov{n}\in \N$, such that for all $[\omega]\in\mathscr R-\mathscr{R}_{\ov{n}}$ 
we have
\[ {\rm Ch}_{\L_{[\omega]}}(X_{[\omega]})= \left((\theta_{F^{ar}/F,S,v_0})_{[\omega]}\right).\]
\end{thm}

\noindent
This provides evidence for the formulation of a Main Conjecture at least
for class groups of global function fields (Section \ref{sss:2stck}). Moreover, since in the function field case the $\ell$-adic $L$-function
has a unique source both for $\ell\neq p$ and $\ell=p$, it also suggests the 
possibility of looking for some unified treatment of all the $\ell$-parts
of Iwasawa modules. 
For now this, and the problem of formulating a Main Conjecture for other
arithmetically significant $\L$-modules, are just speculations. 
They could be interesting topics for future research.

\bigskip
At the end of this introduction we would like to mention that different
Iwasawa algebras have been recently considered by other authors. 
In particular D.~Burns and A.~Daoud, in \cite{BD}, study the structure 
of modules over the algebra $\Z[[\Z_p]]$ and, together with 
D.~Liang, they also study $I_\bullet$-complete modules over other 
non-noetherian Iwasawa algebras (see \cite{BDL}, where they generalize 
the results for $\Z_p^\infty$-extensions of function fields obtained 
in \cite{ABBL} and \cite{BBLNY}). We thank Prof. Burns for showing us early 
versions of these papers. 

We believe that our methods can also be used for non-abelian extensions (as the ones studied in \cite{Oz}).

\section{The Iwasawa algebra}\label{s:IwaAlg}
For an abelian group $A$, we let $A^\vee$ denote its Pontryagin dual
\[ A^\vee:=\Hom_{cont}(A,\{z\in\C:|z|=1\}).\] 
Since we will only work with (infinite) profinite groups, we can 
identify $A^\vee$ with 
$\Hom_{cont}(A,\bmu)$, where $\bmu\subset\ov\Q$ is 
the set of roots of unity, endowed with the discrete topology. The action of 
$G_\Q:=\Gal(\ov\Q/\Q)$ on $\bmu$ induces an action on $A^\vee$
defined by $(\sigma\cdot\omega)(a)=\sigma(\omega(a))$.

We fix, once and for all, a prime $\ell$ and an embedding 
$\ov\Q\iri \ov\Q_\ell $, which allows us to think of maps in $A^\vee$ as 
taking values in $\ov\Q_\ell$. It also induces an embedding 
$G_{\Q_\ell}:=\Gal(\ov\Q_\ell/\Q_\ell)\iri G_\Q$, hence an action of
$G_{\Q_\ell}$ on $A^\vee$.

\subsection{Structure theorem} Fix another prime $p\neq \ell$, and let
\[ \G=\varprojlim \G_n \] 
be a profinite abelian $p$-group. We set
\[ \Lambda:=\Z_\ell[[\G]]=\varprojlim \Z_\ell\left[\G_n\right] ,\]
and refer to it as the {\em $\ell$-adic Iwasawa algebra} (or simply the 
{\em Iwasawa algebra}) of $\G$. Each $\Z_\ell[\Gamma_n]$ has a natural
topology as a free $\Z_\ell$-module of finite rank and this induces 
a topological ring structure on $\L$.
\medskip

For any $m\geqslant n$, let $\pi_n^m:\G_m\twoheadrightarrow\G_n$ (respectively, 
$\pi_n:\G\twoheadrightarrow\G_n$) denote the maps defining (respectively,
arising from) the inverse limit. In the arithmetic applications we have in 
mind, $\G$ will always be a Galois group. We can assume that all maps 
$\pi^m_n$ are surjective and usually call them projections. 
As a consequence, the projections $\pi_n$ are also surjective.

Dualizing, the projections $\pi^m_n$ and $\pi_n$ induce natural 
inclusion homomorphisms 
$\G_n^\vee\hookrightarrow\G_m^\vee$ and $\G_n^\vee\hookrightarrow\G^\vee$, 
which allow us to think of $\G^\vee$ as $\bigcup_n\G_n^\vee$. 

By abuse of notation, we shall also use  $\pi_n^m$ and $\pi_n$ to denote 
the maps induced on the respective group algebras. In the same way, if 
$\omega$ is a character in $\G^\vee_n\subseteq\G^\vee$, we shall write 
$\omega$ for both ring homomorphisms $\Z_\ell[\G_n]\rightarrow\ov\Q_\ell$ and 
$\L\rightarrow\ov\Q_\ell$ obtained extending $\omega$.

Moreover, we write $\mathscr{R}:=\G^\vee/G_{\Q_\ell}$ (respectively, 
$\mathscr{R}_n:=\G_n^\vee/G_{\Q_\ell}$) for the set of orbits under the 
action of $G_{\Q_\ell}$ on $\G^\vee$ (respectively, on $\G_n^\vee$).
We remark that the image $\omega(\Gamma)$ is independent of the choice of the 
representative in the orbit $[\omega]$, i.e., for any $[\omega]\in \mathscr{R}$ and any 
$\chi \in [\omega]$, we have $\omega(\G)=\chi(\G)$.

\begin{thm} \label{t:struttura1} %Assume $d\in \N_+$, then
We have
\begin{equation} \label{e:struttura} 
\Lambda=\prod_{[\omega]\in\mathscr{R}} \L_{[\omega]} ,
\end{equation}
where $\L_{[\omega]}\subseteq\Lambda$ is a ring (non-canonically) isomorphic to $\Z_\ell[\omega(\Gamma)]$.
\end{thm}

\begin{proof} Let $\ov{\Z}_\ell$ be the integral closure of $\Z_\ell$ into $\ov{\Q}_\ell$. 
The group $G_{\Q_\ell}$ acts on the coefficients of 
$\ov{\Z}_\ell[\G_n]$ inducing a $\Z[[G_{\Q_\ell}]]$-module structure. 

\noindent
Obviously, $\Z_\ell[\G_n]=\ov{\Z}_\ell[\Gamma_n]^{G_{\Q_\ell}}$.

For any $\omega\in\Gamma^\vee=\bigcup_n \G_n^\vee$, 
we define $e_{\omega,n}\in\ov\Z_\ell[\Gamma_n]$ by
\[ e_{\omega,n}:=\left\{ \begin{array}{ll} \displaystyle\frac{1}{|\Gamma_n|}
\sum_{\gamma\in\Gamma_n}\omega(\gamma^{-1})\gamma & 
\text{ if }\omega\in\Gamma_n^\vee,\\ 
\ & \\
0 & \text{ otherwise},\end{array} \right.\]
and set
\[ e_{[\omega],n}:=\sum_{\chi\in[\omega]}e_{\chi,n} . \]
Then, $e_{[\omega],n}\in\Z_\ell[\Gamma_n]$, because it is Galois invariant. 
A standard computation shows that $\{e_{\omega,n}:\omega\in\Gamma_n^\vee\}$ is a set 
of orthogonal idempotents with sum $1$, corresponding to the decomposition of 
$\ov\Z_\ell[\Gamma_n]$ in isotypic components. Therefore, the relations 
\begin{equation}\label{EqIdempn} 
\sum_{[\omega]\in\mathscr{R}_n}e_{[\omega],n}=1\,\ \text{ and }\,\ 
e_{[\omega],n}\cdot e_{[\chi],n}=\left\{\begin{array}{ll} e_{[\omega],n} & 
\text{ if } [\omega]=[\chi] ,\\ 
\ & \\
0 &  \text{ if } [\omega]\neq[\chi] \end{array} \right. 
\end{equation}
hold in $\Z_\ell[\G_n]$ for all $n$, and all $[\omega], [\chi]\in \mathscr{R}_n$. 

\noindent Moreover, if $\omega\in\Gamma_m^\vee$ and $m\geqslant n$, we have
\begin{align*} 
\pi_n^m(e_{\omega,m}) & =  
\pi_n^m\left(\frac{1}{|\Gamma_m|}\,
\sum_{\gamma\in\Gamma_m}\omega(\gamma^{-1})\gamma\right)
 = \frac{1}{|\Gamma_m|}\,
\sum_{\gamma\in\Gamma_m}\omega(\gamma^{-1})\pi_n^m(\gamma)\\
& = \frac{1}{|\Gamma_m|}\,
\sum_{\delta\in\Gamma_n}\left(\sum_{\gamma\in(\pi_n^m)^{-1}(\delta)}
\omega(\gamma^{-1})\right)\delta = e_{\omega,n}. 
\end{align*}
To prove the last equality, just write $(\pi_n^m)^{-1}(\delta)=
\gamma_\delta\Ker(\pi_n^m)$ 
for some $\gamma_\delta\in\Gamma_m$ such that $\pi^m_n(\g_\delta)=\delta$.
This leads to
\begin{align*} 
\sum_{\gamma\in(\pi_n^m)^{-1}(\delta)}\omega(\gamma^{-1}) & =
\sum_{\eta\in\Ker(\pi_n^m)} \omega(\gamma_\delta^{-1})\,\omega(\eta^{-1}) \\
\ & =\left\{ \begin{array}{ll} |\Ker(\pi_n^m)|\,\omega(\g_\delta^{-1}) = 
\displaystyle{\frac{|\G_m|}{|\G_n|}\,\omega(\delta^{-1})} & 
\text{ if }\omega\in\Gamma_n^\vee, \\ 
\ & \\
0 & \text{ otherwise}.   \end{array} \right.
\end{align*}
Thus, for any $[\omega]\in \mathscr{R}$, we obtain a coherent sequence
\[ e_{[\omega]}:=\left(e_{[\omega],n}\right)_n \in 
\varprojlim \Z_\ell[\G_n]=\Lambda \]
such that $\pi_n(e_{[\omega]})=e_{[\omega],n}$ for every $n$.

Note that the infinite sum $\sum_{[\omega]}e_{[\omega]}$ converges to 1, 
because $\pi_n(e_{[\omega]})=0$ for all $n$ such that $[\omega]\not\in\mathscr{R}_n$, and 
thus, the generic summand tends to $0$ in $\L$. Hence one can take limits in 
\eqref{EqIdempn} to obtain analogous relations in $\L$. Therefore, the (numerable) set 
$\{e_{[\omega]} : [\omega]\in \mathscr{R}\}$ consists of orthogonal idempotents 
with sum 1 in $\L$.

\noindent
The decomposition \eqref{e:struttura} immediately follows by putting 
\[ \L_{[\omega]}:=e_{[\omega]}\L .\]

As for the last statement, fix a representative $\omega$ for an orbit 
$[\omega]\in \mathscr{R}$. Extending by linearity $\omega:\G\rightarrow\ov\Q_\ell^{\,*}$ 
to a ring homomorphism $\omega:\L\rightarrow\ov\Q_\ell$, we have 
$\omega(\L)=\Z_\ell[\omega(\G)]$. Besides, orthogonality of characters implies
$\omega(e_{\chi,n})=0$ if $\chi,\omega\in\G^\vee$ are distinct, yielding
$\omega(\L_{[\chi]})=0$ for every $[\chi]\neq[\omega]$. 
This establishes a surjective homomorphism
\begin{equation} \label{e:omgism} 
\omega\colon\L_{[\omega]}\longrightarrow\Z_\ell[\omega(\G)] , 
\end{equation} 
which is both a ring and a $\Z_\ell$-module homomorphism.

Since $\omega\in\G_n^\vee$ for some $n$, we have 
\[ \L_{[\omega]}=e_{[\omega]}\L\simeq e_{[\omega],n}\Z_\ell[\G_n]\subseteq\Z_\ell[\G_n].\]
Hence, $\L_{[\omega]}$ is a free $\Z_\ell$-module of finite rank. 
The same holds for $\Z_\ell[\omega(\G_n)]$, and, in order to prove that \eqref{e:omgism} is an 
isomorphism, it suffices to check $\Z_\ell$-ranks of both sides. 

Let $\Stab(G_{\Q_\ell},\omega)$ be the stabilizer of 
$\omega$ in $G_{\Q_\ell}$. Obviously, $\Stab(G_{\Q_\ell},\omega)$ 
is contained in $\Gal(\ov\Q_\ell/\Q_\ell(\omega(\G)))$.  
Conversely, $\tau\in \Gal(\ov\Q_\ell/\Q_\ell(\omega(\G)))$ yields 
$\tau(\omega(\g))=\omega(\g)$ for all $\g\in \G$, i.e., 
$\tau\omega=\omega$ and $\tau\in \Stab(G_{\Q_\ell},\omega)$.  Therefore 
\[\Stab(G_{\Q_\ell},\omega) = \Gal(\ov\Q_\ell/\Q_\ell(\omega(\G))).\]
In particular the stabilizer is always a normal subgroup of $G_{\Q_\ell}$.
Thus,
\[ \rank_{\Z_\ell}\Z_\ell[\omega(\G)]=[\Q_\ell(\omega(\G)):\Q_\ell]=
\left| G_{\Q_\ell} / \Stab(G_{\Q_\ell},\omega) \right|= 
\left| [\omega]\right| . \]
We conclude by noting that
\[ \rank_{\Z_\ell}\L_{[\omega]}=
\rank_{\ov\Z_\ell}(\L_{[\omega]}\otimes\ov\Z_\ell)=
\rank_{\ov\Z_\ell}\bigoplus_{\chi\in[\omega]}e_{\chi,n}\ov\Z_\ell[\G_n]
=\left|[\omega]\right| ,\]
where the last equality follows from the fact that isotypic components 
of the regular representation of a finite abelian group have rank $1$.
\end{proof}

\begin{rems} \label{RemLambdaCompact} \  
\begin{enumerate}[{\bf 1.}]
\item The right-hand side of \eqref{e:struttura} is an {\em inner} direct
product, since each ring $\L_{[\omega]}$ is a submodule of $\L$. Note that 
it could also be written as $\sum_{[\omega]}\L_{[\omega]}$, 
since, as we observed in the proof, the $e_{[\omega]}$ converge to $0$ in $\L$. 
We shall choose the more convenient notation among 
$x=(x_{[\omega]})_{[\omega]}$ and 
$x=\sum_{[\omega]}e_{[\omega]}x$ for $x\in\L$ according to the situation.

\smallskip
\noindent
\item The topology on $\L$ is the same as the product topology induced 
by \eqref{e:struttura}. Indeed, the $\L_{[\omega]}$ are finitely generated 
(free) $\Z_\ell$-modules (hence compact), and the product on the right 
is compact by the Tychonov Theorem. 
Moreover, $\L$ is compact, because it is the inverse limit of the 
finitely generated $\Z_\ell$-modules $\Z_\ell[\G_n]$. Finally, the 
map between $\L$ and the product is induced by
a choice of the continuous characters $\omega$, hence it is a 
continuous bijection between compact spaces, i.e., a homeomorphism. 

\smallskip
\noindent
\item The ring $\Z_\ell[\omega(\G)]$ consists exactly of the integers of 
$\Q_\ell(\omega(\G))$. 

\smallskip
\noindent
\item The proof of Theorem \ref{t:struttura1} also yields isomorphisms
\[\prod_{[\omega]\in\mathscr R_n}\L_{[\omega]}\;
\stackrel{\,\sim\ }{\longrightarrow}\;\Z_\ell[\G_n]=:\Lambda_n\]
induced by the projections $\pi_n$ for every $n$.
\end{enumerate} 
\end{rems}

\subsection{Finitely generated ideals}
For any $[\omega]\in \mathscr{R}$, let $\pi_{[\omega]}\colon \L 
\longrightarrow \L_{[\omega]}$ be the canonical projection 
$x\mapsto e_{[\omega]}x$. 

The field $\Q_\ell(\omega(\G))$ is an unramified extension of 
$\Q_\ell$ (because $p\neq\ell$), so 
Remark \ref{RemLambdaCompact}.3 implies that $\L_{[\omega]}$ is a 
discrete valuation ring whose maximal ideal is generated by
\[ \ell_{[\omega]}:=e_{[\omega]}\ell.\]

\begin{prop}\label{LambdaPI}
An ideal in $\L$ is principal if and only if it is closed.

\noindent
Moreover, a closed ideal $I\neq (0)$ is generated by $\alpha_I\in\L$ such that
\begin{equation} \label{e:alfaI} 
e_{[\omega]}\alpha_I=\ell_{[\omega]}^{n(I,\omega)}\;\;
\text{ where }\  n(I,\omega):=
\min\left\{v_{\ell_{[\omega]}}(x)\,:\,x\in \pi_{[\omega]} (I)\right\}
\end{equation}
(with the usual convention $\ell_{[\omega]}^{+\infty}=0$).
\end{prop}
  
\begin{proof} Let $I$ be an ideal of $\Lambda$. 
Then $\pi_{[\omega]}(I)=e_{[\omega]}I$ is an ideal of $\Lambda_{[\omega]}$
and so it is generated by a power of $\ell_{[\omega]}$, that is, 
$e_{[\omega]}I=(\ell_{[\omega]}^{n(\omega,I)})$ with $n(\omega,I)$ as 
in \eqref{e:alfaI}. The sequence 
\[ \alpha_n=\sum_{[\omega]\in\mathscr R_n}\ell_{[\omega]}^{n(\omega,I)} \] 
takes values in $I$ and converges to $\alpha_I$. Therefore, 
$(\alpha_I)\subseteq I$ is true if $I$ is closed.

To prove the opposite inclusion, it is enough to note that for 
$\beta\in I$ one has 
$e_{[\omega]}\beta=x_{[\omega]}\ell_{[\omega]}^{n(\omega,I)}$ for 
some $x_{[\omega]}\in\Lambda_{[\omega]}$ and hence
\[ \beta=\sum_{[\omega]\in\mathscr{R}} x_{[\omega]}\ell_{[\omega]}^{n(\omega,I)}
=x\alpha_I \]
with $x=\sum_{[\omega]} x_{[\omega]}\in\L$. This proves $I=(\alpha_I)$.

\smallskip
Conversely, any principal ideal $I=\alpha\Lambda$ is closed because
multiplication by $\alpha$ is continuous and $\Lambda$ is Hausdorff 
and compact.
\end{proof}

\begin{cor}\label{LambdaFGPI}
Every finitely generated ideal in $\L$ is principal.
\end{cor}

\begin{proof} 
Since $\L$ is compact, every finitely generated ideal is closed, 
hence principal by the previous proposition.
\end{proof}

\begin{rem}
Another immediate consequence of Proposition \ref{LambdaPI} is the
fact that the closed ideals of $\L$ can be identified with functions 
$\mathscr{R}\rightarrow\N\cup\{+\infty\}$.
\end{rem}

It is worth noting that $\Lambda$ is Noetherian if and only if 
$\mathscr{R}$ is finite, since the finitely generated condition 
only holds (in general) for closed ideals. However,  
its structure naturally leads to a definition of characteristic ideals 
for locally finitely generated modules, 
even when $\G$ is not topologically finitely generated (this latter case, 
on the contrary, needs some special treatment in the case of the algebra 
$\Z_p[[\G]]$, see, e.g., \cite{BBLNY} and \cite{BDL})

\noindent
This is particularly relevant in the Iwasawa theory framework, 
where characteristic ideals provide the algebraic ingredient of 
Main Conjectures.

\subsection{The spectrum of $\L$}  
An immediate consequence of Theorem \ref{t:struttura1} 
is the fact that prime ideals of $\L$ have the form
\[(0_{[\omega]})\oplus(1-e_{[\omega]})\L\]
or
\[\gotm_{[\omega]}\oplus(1-e_{[\omega]})\L\]
where $\gotm_{[\omega]}=(\ell_{[\omega]})$ is the maximal ideal of 
$\L_{[\omega]}$. It follows that the nilradical of $\L$ is $(0)$, and 
that the Jacobson radical of $\L$ is
\[\prod_{[\omega]\in\mathscr{R}} \gotm_{[\omega]}=
\prod_{[\omega]\in\mathscr{R}}(\ell_{[\omega]})=\ell\L .\]
In particular, $\L$ has Krull dimension $1$, and 
\begin{equation} \label{e:spettro} 
\Spec(\L)= \bigsqcup_{[\omega]\in\mathscr{R}}\Spec(\L_{[\omega]}).
\end{equation} 
Localizing at $\gotm_{[\omega]}\oplus(1-e_{[\omega]})\L$, we obtain 
$\L_{[\omega]}$ itself, while the localization of $\L$ at 
$(0_{[\omega]})\oplus(1-e_{[\omega]})\L$ is the 
quotient field $Q(\L_{[\omega]})\simeq\Q_\ell(\omega(\G))$. Hence we 
can interpret $\L$ also as the direct product of its localization at 
maximal ideals (equivalently, at height 1 primes).

\subsection{Augmentation ideals}\label{SecAugId}
Recall that $\G=\displaystyle{\varprojlim \G_n}$, and we have canonical
projections $\pi_n:\G \twoheadrightarrow \G_n$. We set $G_n=\Ker(\pi_n)$,
so that $\G_n\simeq \G/G_n$.

\begin{dfn}\label{DefAugId}
For every $n$, let 
$I_n:= \big(\g-1\,:\,\g\in G_n\big)\subseteq\L$ be the 
{\em augmentation ideal associated with $G_n$}. 
\end{dfn}

By definition,
\[ \L = \varprojlim \Z_\ell [\G_n] \simeq \varprojlim \Z_\ell[\G/G_n]
\simeq \varprojlim \L/I_n\L ,\]
and 
\[ \L/I_n\L \simeq \Z_\ell[\G_n] = \bigoplus_{[\omega]\in \mathscr{R}_n} 
e_{[\omega],n} \Z_\ell[\G_n] \simeq \bigoplus_{[\omega]\in \mathscr{R}_n}
\L_{[\omega]} .\]
Therefore, for any $n$, we have
\[ I_n = \bigcap_{[\omega]\in \mathscr{R}_n} (1-e_{[\omega]})\L \simeq
 \prod_{[\omega]\in \mathscr{R}-\mathscr{R}_n} \L_{[\omega]} ,\]
and all augmentation ideals are idempotent, i.e., $I_n^2=I_n$
(this yields a simpler proof of \cite[Lemma 3.7]{BL1}).

A notable consequence is the fact that, contrary to the $\ell=p$ case, 
augmentation ideals cannot be used to apply the usual topological
generalization of Nakayama's Lemma (\cite[Section 3]{BaHo}),
because they are not topologically nilpotent.

\section{Characteristic ideals of $\L$-modules}\label{s:CharId}
The path indicated by Theorem \ref{t:struttura1} and the structure of $\L$
 naturally lead to the definition of an analogue of characteristic ideals 
 for certain $\L$-modules.

\subsection{Locally finitely generated $\L$-modules}
In accordance with \eqref{e:spettro}, a $\L$-module $M$ is equivalent 
to a sheaf of $\L_{[\omega]}$-modules 
$(M_{[\omega]})_{[\omega]\in\mathscr{R}}$ on
$\Spec \L$, defined by the components $M_{[\omega]}=e_{[\omega]}M$. 

\begin{dfn}\label{DefOmegaComp}
Let $M$ be a $\L$-module. Then, for every character 
$[\omega]\in \mathscr{R}$, we define the {\em $[\omega]$-component} 
of $M$ as the $\L_{[\omega]}$-module 
\[ M_{[\omega]} := e_{[\omega]} M . \]
We say that a $\L$-module $M$ is {\em locally finitely generated} 
if each $M_{[\omega]}$ is a finitely generated $\L_{[\omega]}$-module.
 
\noindent
We say that a locally finitely generated $\L$-module $M$ is 
{\em locally torsion} if each $M_{[\omega]}$ is a finitely generated 
torsion $\L_{[\omega]}$-module.
\end{dfn}

Note that being locally finitely generated is not enough for a module 
$M$ to be finitely generated over $\L$ (see Example \ref{eg:infgen}).

\subsubsection{Character decomposition of $M_{[\omega]}$}\label{ss:CharDec} The following will be useful to formulate analogues of the Main Conjecture.

\begin{prop} \label{p:cardecM} Let $M$ be a $\Lambda$-module. Its $[\omega]$-component has the decomposition
\begin{equation} \label{e:cardecM} M_{[\omega]}\otimes_{\Z_\ell}\Z_\ell[\omega(\Gamma)]=\bigoplus_{\chi\in[\omega]}M_{[\omega]}^{(\chi)} \end{equation}
where $M_{[\omega]}^{(\chi)}$ is the $\Z_\ell[\omega(\Gamma)]$-submodule on which $\Gamma$ acts via the character $\chi$.
Moreover $\Gal(\Q_\ell(\omega(\Gamma))/\Q_\ell)$ acts permuting transitively the $\chi$-components. In particular, one has
\begin{equation} \label{e:ismrfch} M_{[\omega]}^{(\chi)}\simeq M_{[\omega]}^{(\chi')} \end{equation}
as $\Z_\ell$-modules, if $\chi$ and $\chi'$ are both in $[\omega]$. \end{prop}

\begin{proof} In order to obtain \eqref{e:cardecM}, take $n$ such that $\omega$ factors through $\Gamma_n$. Then we can consider the operator
\[\eta_{\chi,n}:=\frac{1}{|\Gamma_n|}\sum_{\gamma\in\Gamma_n}\gamma\otimes\chi(\gamma^{-1})\]
and define
\[ M_{[\omega]}^{(\chi)}:=\eta_{\chi,n}\big(M_{[\omega]}\otimes_{\Z_\ell}\Z_\ell[\omega(\Gamma)]\big).\]
One checks that this is independent of the choice of $n$ by reasoning as in the proof of Theorem \ref{t:struttura1}. 

For any pair $\chi,\chi'\in[\omega]$, there is $\sigma\in\Gal(\Q_\ell(\omega(\Gamma))/\Q_\ell)$ such that $\chi'=\sigma\cdot\chi$. But then $1\otimes\sigma$ yields \eqref{e:ismrfch}, because $\eta_{\chi',n}=(1\otimes\sigma)(\eta_{\chi,n})$.    
\end{proof}

\subsection{Characteristic ideals}\label{SecCharId}
In classical Iwasawa theory, characteristic ideals are the 
algebraic counterpart of $p$-adic $L$-functions. 
In this section we provide a serviceable definition for 
the characteristic ideal of locally finitely generated $\L$-modules.

Assume that $M$ is a locally finitely generated $\L$-module. 
The structure theorem for finitely generated modules over discrete 
valuation rings yields isomorphisms 
\begin{equation} \label{e:struttura2} 
M_{[\omega]} \simeq \L_{[\omega]}^{r_{[\omega]}}\oplus 
\bigoplus_{j=1}^{s_{[\omega]}} \L_{[\omega]}/I_{[\omega],j} ,  
\end{equation}
where the non-negative integers $r_{[\omega]}$ and $s_{[\omega]}$, and 
the finite sets of ideals $I_{[\omega],j}$ are all uniquely determined 
by $M$ and $\omega$.

\begin{dfn}\label{DefCharId}
Let $M$ be a locally finitely generated $\L$-module. Assume 
every $M_{[\omega]}$ is written as in \eqref{e:struttura2}. 
Then, for any $[\omega]\in \mathscr{R}$, the {\em characteristic ideal} of 
 $M_{[\omega]}$ (or the {\em $[\omega]$-characteristic ideal} of $M$) is
\[ {\rm Ch}_{\L_{[\omega]}} (M_{[\omega]})=\left\{\begin{array}{ll}
\ds{ \prod_{j=1}^{s_{[\omega]}} I_{[\omega],j} }& 
\text{ if } r_{[\omega]}=0 ,\\
\ & \\
(0_{[\omega]}) & \text{ if } r_{[\omega]}\neq 0. 
\end{array}  \right. \] 
The {\em characteristic ideal} of $M$ is
\[ {\rm Ch}_\L (M) := 
\left( {\rm Ch}_{\L_{[\omega]}} (M_{[\omega]})\right)_{[\omega]}
\subseteq \L .\]
In particular, ${\rm Ch}_{\L_{[\omega]}}(M_{[\omega]})=\pi_{[\omega]}
({\rm Ch}_\L(M))$ for all $[\omega]\in\mathscr{R}$.
\end{dfn}

\begin{rem}\label{RemCharId}
The characteristic ideal in $\L$ of a locally finitely generated 
$\L$-module $M$ is $(0)$ if and only if all its 
$[\omega]$-components are non-torsion. Indeed, 
$\pi_{[\omega]}({\rm Ch}_\L(M))$ is non-trivial in $\Lambda_{[\omega]}$ 
if and only if $M_{[\omega]}$ is $\L_{[\omega]}$-torsion. In particular, 
$\pi_{[\omega]}({\rm Ch}_\L(M))$ is non-trivial
for all $[\omega]\in \mathscr{R}$ if and only if $M$ is locally torsion.

\noindent
For example, the characteristic ideal of the module $\L_{[\omega]}$ is 
\[ {\rm Ch}_\L(\L_{[\omega]})=(1-e_{[\omega]})\L , \]
i.e.,
\[ \pi_{[\chi]}({\rm Ch}_\L(\L_{[\omega]})) =
\left\{ \begin{array}{ll} (1_{[\chi]})=\L_{[\chi]} & 
\text{ if } [\chi]\neq [\omega] ,\\
\ & \\
(0_{[\omega]}) & \text{ if } [\chi]=[\omega]. \end{array} \right.  \]
\end{rem}

\begin{prop}\label{CharIdSeq}
Let 
\[ 0\longrightarrow M \longrightarrow N \longrightarrow L 
\longrightarrow 0 \]
be a short exact sequence of locally finitely generated $\L$-modules. Then, 
\[ {\rm Ch}_\L (N) = {\rm Ch}_\L (M){\rm Ch}_\L (L) .\]
\end{prop}

\begin{proof}
The ideals are defined by their components so we need to check that for any 
$[\omega]\in \mathscr{R}$ 
\[ \pi_{[\omega]}({\rm Ch}_\L (N)) = 
\pi_{[\omega]}\left({\rm Ch}_\L(M){\rm Ch}_\L(L)\right) ,\]
i.e.,
\[ {\rm Ch}_{\L_{[\omega]}}(N_{[\omega]})=
{\rm Ch}_{\L_{[\omega]}}(M_{[\omega]})
{\rm Ch}_{\L_{[\omega]}}(L_{[\omega]}) .\]
The latter equality is a trivial consequence of the fact that 
$\L_{[\omega]}$ is a discrete valuation ring with finite residue field. 
\end{proof}

\subsubsection{Extension of constants and Main Conjecture}\label{ss:ExtConstMC} 
The decomposition \eqref{e:cardecM}
involves $\Lambda\otimes_{\Z_\ell}\Z_\ell[\omega(\Gamma)]$-modules.
By (an analogue of) Proposition \ref{CharIdSeq},
\begin{align} \label{e:cardecid}
{\rm Ch}_{\L_{[\omega]}} (M_{[\omega]})\otimes_{\Z_\ell}\Z_\ell[\omega(\Gamma)] &
= \prod_{\chi\in[\omega]} {\rm Ch}_{\L_{[\omega]}\otimes_{\Z_\ell}\Z_\ell[\omega(\Gamma)]} 
(M_{[\omega]}^{(\chi)}) \\
 \ & = \left({\rm Ch}_{\L_{[\omega]}\otimes_{\Z_\ell}\Z_\ell[\omega(\Gamma)]} 
(M_{[\omega]}^{(\omega)})\right)^{|[\omega]|} , \nonumber
\end{align} 
where the last equality follows from \eqref{e:ismrfch}.
Whenever $|[\omega]|>1$ (i.e., when $\omega(\Gamma)$ is not a subset of $\Z_\ell$), 
this gives rise to a phenomenon which has no analogue in the classical Iwasawa theory with $\ell=p$, where decompositions like the ones in 
Theorem \ref{t:struttura1} and Proposition \ref{p:cardecM} do not exist.

To formulate a Main Conjecture for a $\L$-module $M$, we need to define
an $\ell$-adic $L$-function $\mathscr{L}_M$, which interpolates, {\em at all $\chi\in \G^\vee$}, special values of 
$L$-functions attached to $M$. Thus we expect a direct link between $\chi(\mathscr{L}_M)$ and $M_{[\omega]}^{(\chi)}$. On the other hand, a 
Main Conjecture requires a relation between an element $\mathcal{L}_M$ in $\L$ and the characteristic ideal ${\rm Ch}_\L (M)$. Therefore, to comply with
Definition \ref{DefCharId} and with \eqref{e:cardecid}, we shall define our $\mathcal{L}_M$ as the 
unique element of $\L$ whose $[\omega]$-component is the product of all the $\chi$-components of 
$\mathscr{L}_M$ as $\chi$ varies in $[\omega]$. We shall do so in Section \ref{s:StickEl}, where $M=X$
is the class group of a $\Z_p^d$-extension of a global function field,
$\mathscr{L}_M=\Theta_{\calf/F,S,v_0}(u)$ 
(Proposition \ref{p:StickTateAlg} and \eqref{e:interp}) 
and $\mathcal{L}_M=\theta_{\calf/F,S,v_0}$ is in Definition \ref{d:ModStickEl}. 

\begin{rem} When $M$ is locally torsion, all characteristic ideals are just powers of 
$\ell$ (i.e., ideals in $\Z$). Note that this contributes to hide the importance of the underlying 
structure and the contribution of every single character $\chi$. 
\end{rem}

\section{Sinnott modules}\label{SecSinMod}
We now describe a specific class of $\L$-modules which will include all our arithmetic examples. The main feature will be formulas for ranks (respectively, orders) of locally finitely generated (respectively, locally torsion) $\L$-modules. Such formulas shall provide a vast generalization 
of Sinnott's Theorem on the $p$-adic convergence of the orders 
of $\ell$-class groups in $\Z_p$-extensions of number fields. 
Because of this we decided to call them {\em Sinnott modules}.

\subsection{$I_\bullet$-completeness}\label{s:IComplMod}
For any $n\geqslant 0$, set
\[ e_n:=\sum_{[\omega]\in\mathscr{R}_n}e_{[\omega]} ,\]
and 
\begin{equation} \label{e:enqdr} 
e_{[n]}:=\sum_{[\omega]\in\mathscr{R}_n-\mathscr{R}_{n-1}}
e_{[\omega]}=e_n-e_{n-1}\, \end{equation}
where we put $\mathscr{R}_{-1}=\emptyset$ and $e_{-1}=0$.

Let $M$ be a $\L$-module. Clearly, $I_ne_nM=0$. Hence, $e_nM$ has 
the structure of $\L$ and also of $\L/I_n\L$-module, and 
 \[  e_nM\simeq M/I_nM .\] 

\begin{dfn}\label{d:nQuotnComp}
For any $\L$-module $M$ and for any $n$, we define the {\em $n$-quotient} 
of $M$ as 
\[ e_nM\simeq M/I_nM \simeq M \otimes_\L \L/I_n\L  .\]
We also define the {\em $[n]$-component} of $M$ as
\[ M_{[n]} := e_{[n]} M . \]
\end{dfn}

\noindent
By construction, for all $m > n$, we have natural projections 
$M/I_mM \twoheadrightarrow M/I_nM$,
which define an inverse limit.

\noindent
Moreover, by definition, for any $\L$-module $M$, 
\begin{equation}\label{e:DecSinMod}
 M/I_nM\simeq e_nM = \bigoplus_{i=0}^n M_{[i]}. 
\end{equation}

\begin{dfn}\label{DefIComplMod}
A $\L$-module $M$ is said to be {\em $(I_n)_n$-complete} (or simply 
{\em $I_\bullet$-complete}) if the natural projection map
\[ \pi_M : M \longrightarrow \ds{ \liminv M/I_nM} \]
is bijective.
\end{dfn}

\noindent
By definition, we immediately get a 
decomposition for an $I_\bullet$-complete module
\begin{equation}\label{EqM1} 
M \simeq \varprojlim_n\, \bigoplus_{i=0}^n M_{[i]} = 
\prod_{n\geqslant 0} M_{[n]} = 
\prod_{[\omega]\in \mathscr{R}} M_{[\omega]} .
\end{equation}
 
Vice versa, any collection of $\L_{[\omega]}$-modules 
$\{ M_{[\omega]}\}_{[\omega]\in \mathscr{R}}$ defines an 
$I_\bullet$-complete $\L$-module by simply taking 
$M:=\prod_{[\omega]\in \mathscr{R}} M_{[\omega]}$.
Indeed, one has $M_{[\omega]}=e_{[\omega]}M$ for all $[\omega]$, 
and therefore, $I_nM=\prod_{[\omega]\in \mathscr{R}-\mathscr{R}_n}
M_{[\omega]}$. Hence
\begin{align*} 
M:&=\prod_{[\omega]\in \mathscr{R}} M_{[\omega]} =
\prod_{n\geqslant 0}\, 
\bigoplus_{[\omega]\in \mathscr{R}_n-\mathscr{R}_{n-1}} e_{[\omega]}M \\
&=\prod_{n\geqslant 0}\,M_{[n]}  = \varprojlim e_n M = \varprojlim M/I_nM ,
\end{align*}
where all maps are canonical projections as required by 
Definition \ref{DefIComplMod}.

\noindent
Note that the $I_\bullet$-completeness of 
$M=\prod_{[\omega]\in \mathscr{R}} M_{[\omega]}$, is independent of 
the $M_{[\omega]}$. In particular, they do not need to be complete 
with respect to their own topology. For example,  
\[ M=\prod_{[\omega]\in \mathscr{R}} \L_{[\omega]}[x] , \]
is an $I_\bullet$-complete $\L$-module for any variable $x$.

\begin{ex} \label{eg:infgen} 
Let $\Gamma=\Z_p$ and consider the $I_\bullet$-complete module
\[ M=\prod_{[\omega]\in\mathscr{R}}\L_{[\omega]}^{|\omega(\Gamma)|} .\]
Then, $M$ is compact (being a product of compact rings) and locally 
finitely generated. However, $M$ is not a finitely generated $\L$-module. 

\noindent
Indeed, assume  $x_1,\dots,x_k$ generate $M$ as a $\L$-module. 
Then $e_nM$ is a $\L/I_n\L$-module generated by $e_nx_1,\dots,e_nx_k$. 
This leads to a contradiction, because the $\L/I_n\L$-rank of 
$e_nM=\prod_{[\omega]\in\mathscr R_n}\L_{[\omega]}^{|\omega(\Gamma)|}$ 
is unbounded as $n$ grows. 
\end{ex}

Examples like the one above show that being locally finitely generated 
is far from enough to guarantee that a module $M$ is finitely generated 
over $\L$. 

\begin{prop} \label{p:FinGenICMod}
An $I_\bullet$-complete $\Lambda$-module $M$ is finitely generated if 
and only if it is locally finitely generated and the local rank 
function 
\[ \rank_{\ell,M}\colon \mathscr R\longrightarrow\N\,,\quad
[\omega]\mapsto\dim_{\F_{[\omega]}}M_{[\omega]}/\ell_{[\omega]}M_{[\omega]} \]
is bounded (where $\F_{[\omega]}:=\Lambda_{[\omega]}/\gotm_{[\omega]}
= \Lambda_{[\omega]}/\ell_{[\omega]}\Lambda_{[\omega]}$ 
denotes the residue field of 
the local ring $\Lambda_{[\omega]}$).
\end{prop}

\begin{proof} It is obvious that if $x_1,\dots,x_n$ generate $M$ over 
$\Lambda$ then $e_{[\omega]}x_1,\dots,e_{[\omega]}x_n$ generate 
$M_{[\omega]}$ and 
$\dim_{\F_{[\omega]}}M_{[\omega]}/\ell_{[\omega]}M_{[\omega]}\leqslant n$
for every $[\omega]$.
 
Vice versa, if $\rank_{\ell,M}$ is bounded by $n$, then for every 
$[\omega]$ there are elements 
$x_{1,[\omega]},\dots,x_{n,[\omega]}\in M_{[\omega]}$ whose reductions 
modulo $\ell$ generate $M_{[\omega]}/\ell_{[\omega]}M_{[\omega]}$.
 Since 
$M$ is $I_\bullet$-complete, we can take
\[ x_i:=\big(x_{i,[\omega]}\big)\in\prod_{[\omega]}M_{[\omega]}=M \]
for $i\in\{1,\dots,n\}$. One easily checks that $x_1,\dots,x_n$ generate 
$M$ as a $\Lambda$-module.
\end{proof}

\noindent
For other properties of $I_\bullet$-complete modules (in a more 
general setting) see, e.g., \cite[Section 2.3]{BDL}.

\subsection{$\Z_\ell$-rank of $\L_{[\omega]}$}\label{s:lOrd} 
Let
\[ \mathscr S_n:=\big\{[\omega]\in\mathscr R:\omega(\G)=\bmu_{p^n}\big\} \]
Thus, for any $[\omega]\in\mathscr S_n$, the ring $\L_{[\omega]}$ 
is isomorphic to the ring of integers of $\Q_\ell(\bmu_{p^n})$, that is, 
$\L_{[\omega]}\simeq\Z_\ell[\bmu_{p^n}]$. In particular, the rank of 
$\L_{[\omega]}$ as a $\Z_\ell$-module is the degree of the (unramified) 
local extension $\Q_\ell(\bmu_{p^n})/\Q_\ell$, and hence, it coincides 
with the order of $\ell$ in $(\Z/p^n\Z)^*$. We denote this number 
by $f_n(\ell,p)$, or simply by $f_n$ when $\ell$ and $p$ are clearly 
fixed. We gather here, for later use, some well-known formulas for $f_n$.

For $p\neq 2$, let $a(\ell,p):=v_p(\ell^{\,f_1(\ell,p)}-1)$. 
It is easy to see 
\[ f_n(\ell,p)=[\Q_\ell(\bmu_{p^n}):\Q_\ell] = 
\left\{ \begin{array}{ll} 1 & \text{ if } n=0 ,\\
\ & \\
 f_1(\ell,p) & \text{ if } 1\leqslant n\leqslant a(\ell,p) ,\\
\ & \\
f_1(\ell,p)p^{n-a(\ell,p)} & \text{  if } n>a(\ell,p) .\end{array} \right. \]
For $p=2$ we have an obvious shift in exponents, which leads to
\begin{itemize}
\item if $\ell\equiv 1\pmod{4}$ and $v_2(\ell-1) =:a(\ell,2)$ for some 
$a(\ell,2)\geqslant 2$, then
\[ f_0(\ell,2)=\cdots=f_{a(\ell,2)}(\ell,2)=1 \quad\text{and}
\quad f_n(\ell,2)=2^{n-a(\ell,2)}\ \text{ for any } n>a(\ell,2) ;\]
\item if $\ell\equiv 3\pmod{4}$ and $v_2(\ell^{\,2}-1) =:a(\ell,2)$ for 
some $a(\ell,2)\geqslant 2$, then
\[ f_0(\ell,2)=f_1(\ell,2)=1, \quad f_2(\ell,2)= \cdots=
f_{a(\ell,2)}(\ell,2)=2 \] and
\[ f_n(\ell,2)=2^{n-a(\ell,2)+1}\quad \text{for any } n>a(\ell,2).\]
\end{itemize}
We also note that all of this can be summarized by the statement that 
for every $\ell$ and $p$ one has
\begin{equation} \label{e:fngnrl} 
f_n(\ell,p)=c(\ell,p)p^n\quad \text{for }n\gg 0 
\end{equation}
with $c(\ell,p)\in\Q^*$.

\begin{rem}
If $\G$ is topologically finitely generated, then the sets
$\mathscr S_n$ are all finite and for every $n$ there is some 
$k\geqslant n$ such that $\mathscr S_n\subseteq\mathscr R_k$. Moreover, 
one has $\mathscr S_n=\mathscr{R}_n-\mathscr{R}_{n-1}$ for all $n$ 
assuming $\Gamma_n=\Gamma/\Gamma^{p^n}$.  
In this case, by definition,
\[ e_{[n]}\L\simeq\Z_\ell[\bmu_{p^n}]^{|\mathscr{S}_n|} .\]
\end{rem}
 
At the end of the proof of Theorem \ref{t:struttura1} we have seen that
$|[\omega]|=[\Q_\ell(\omega(\G)):\Q_\ell]$, hence, for all $n\geqslant 0$
\begin{equation} \label{e:snfn=} 
|\mathscr{S}_n|f_n(\ell,p) = \big|\{ \omega\in \G^\vee:\omega(\G)=
\bmu_{p^n}\}\big| .
\end{equation}
In particular, $|[\omega]|=1$ (i.e., $\omega$ is its own equivalence class)
if and only if $\omega(\G)\subset \Z_\ell$. 

Moreover we remark that $[\omega_1]=[\omega_2]$ yields 
$\Ker(\omega_1)=\Ker(\omega_2)$, but the reverse implication does not hold 
because the action of $G_{\Q_\ell}$ on primitive $p^n$-th roots of unity is 
not transitive in general. It is transitive if and only if 
$f_n(\ell,p)=\varphi(p^n)$. This last condition is equivalent to
\begin{enumerate}
\item[a.] $\ell$ is a generator of $(\Z/p\Z)^*$, i.e., $f_1(\ell,p)=p-1$, and
\item[b.] $v_p(\ell^{p-1}-1)=1$, i.e., $a(\ell,p)=1$.
\end{enumerate}

\subsection{Sinnott Modules: orders, $\Z_\ell$-ranks 
and $\ell$-ranks}\label{SecSinMod2}

\begin{dfn}\label{d:SinMod}
We say that a $\L$-module $M$ is a {\em Sinnott module} when $M$ is 
locally finitely generated, $I_\bullet$-complete and the set
$$\mathscr S_n(M):=\big\{[\omega]\in\mathscr S_n:M_{[\omega]}\neq0\big\}$$
is finite for every $n$.

\noindent
Moreover, we say that a Sinnott module is {\em torsion} 
when it is locally torsion.
\end{dfn}

We also define 
\[ \varepsilon_n:=\sum_{[\omega]\in\mathscr S_n}e_{[\omega]} .\]
Then one can see that an $I_\bullet$-complete $\Lambda$-module $M$ is 
Sinnott if and only if $\varepsilon_nM$ is a finitely generated 
$\Z_\ell$-module for every $n$. 

\noindent
If $\G$ is topologically finitely generated, then $\mathscr S_n$ is 
finite and $M$ is Sinnott if it is $I_\bullet$-complete and locally 
finitely generated.
We recall that the latter condition suffices for attaching to $M$ 
the characteristic ideal ${\rm Ch}_\L(M)$ (Definition \ref{DefCharId}).
\medskip 

Let $M$ be a Sinnott module. We attach to it the following numbers:
\begin{equation} \label{e:dfnrnrhn} 
r_n=r_n(M) := \frac{\rank_{\Z_\ell} \varepsilon_nM}{f_n(\ell,p)} \ \  
\text{ and }\ \ \rho_n=\rho_n(M) := 
\frac{\rank_\ell\varepsilon_nM}{f_n(\ell,p)}.
\end{equation}
Moreover, if $M$ is a torsion Sinnott module, i.e., if $r_n=0$ for all 
$n$, we also consider
\begin{equation} \label{e:dfntn}  
t_n=t_n(M):=\frac{v_\ell (|\varepsilon_nM|)}{f_n(\ell,p)} .
\end{equation}

\begin{thm}\label{t:AdmMod1&2} 
Let $M$ be a Sinnott module. Then the numbers $r_n$, $\rho_n$ and (if $M$ 
is torsion) $t_n$ are all in $\N$. Put
\begin{equation} \label{e:dfnM(n)} 
M(n):=\bigoplus_{m=0}^n \varepsilon_mM. 
\end{equation}
With the notation of Section \ref{s:lOrd}, 
if $p\neq 2$, then, for every $n\geqslant 0$,
\[ \rank_{\Z_\ell} M(n) = \sum_{m=0}^n f_m(\ell,p) r_m  =
r_0 + f_1(\ell,p)\left( \sum_{m=1}^{a(\ell,p)}  r_m + 
\sum_{m=a(\ell,p)+1}^n p^{m-a(\ell,p)}r_m \right) .\]
If $p=2$ and $\ell\equiv 1\pmod{4}$, then, for every $n\geqslant 0$,  
\[ \rank_{\Z_\ell} M(n)  = \sum_{m=0}^n f_m(\ell,2) r_m  =
\sum_{m=0}^{a(\ell,2)} r_m + 
\sum_{m=a(\ell,2)+1}^n 2^{m-a(\ell,2)}r_m  .\]
Finally, if $p=2$ and $\ell\equiv 3\pmod{4}$, then,
for every $n\geqslant 0$, 
\[ \rank_{\Z_\ell} M(n) = \sum_{m=0}^n f_m(\ell,2)r_m =
r_0+r_1+2 \left( \sum_{m=2}^{a(\ell,2)} r_m + 
\sum_{m=a(\ell,2)+1}^n 2^{m-a(\ell,2)} r_m \right) .\]
The formulas for $\rank_\ell M(n)$ and (if $M$ is torsion) for 
$v_\ell (|M(n)|)$ are basically identical: 
just replace all values $r_m$ with the numbers $\rho_m$ or\, $t_m$ defined above.
\end{thm}

\begin{proof} By the structure theorem for finitely generated modules 
over principal ideal domains, for any $n\geqslant 0$, we have
\begin{equation}\label{EqMomega0}
\varepsilon_nM=\bigoplus_{[\omega]\in \mathscr S_n}\, M_{[\omega]} \simeq 
\bigoplus_{[\omega]\in \mathscr S_n}\,\left( \L_{[\omega]}^{r_{[\omega]}}
\, \oplus \bigoplus_{j=1}^{s_{[\omega]}}\, 
\L_{[\omega]}/\left(\ell_{[\omega]}^{t_{[\omega], j}}\right) \right)
\end{equation}
where the integers $r_{[\omega]}$, $s_{[\omega]}$ and $t_{[\omega],j}$ 
are uniquely determined by $M$ and $\omega$. Remembering that 
for $[\omega]\in\mathscr S_n$ one has 
$\Lambda_{[\omega]}\simeq\Z_\ell[\bmu_{p^n}]$, we can rearrange the sums 
on the right-hand side of \eqref{EqMomega0} to get
\begin{equation}\label{EqMomega}
\varepsilon_nM \simeq \Z_\ell[\bmu_{p^n}]^{\sum_{[\omega]\in \mathscr S_n} 
r_{[\omega]}}\,\oplus \bigoplus_{h=1}^{s_n}\, 
\Z_\ell[\bmu_{p^n}]/\left(\ell^{t_{n,h}}\right) 
\end{equation}
with $s_n: =\ds{\sum_{[\omega]\in \mathscr S_n} s_{[\omega]}}$ and 
$\big\{t_{n,h}\big\}_h=\big\{t_{[\omega],j}\}_{[\omega],j}$. 
From \eqref{EqMomega} it follows (with $f_n=f_n(\ell,p)$)
\[ \rank_{\Z_\ell}\varepsilon_nM = 
f_n\sum_{[\omega]\in \mathscr S_n} r_{[\omega]}=f_n r_n  \]
and
\[ \rank_\ell \varepsilon_nM = f_n r_n  + f_n s_n= f_n \rho_n  \]
for every $n\geqslant 0$. Moreover, if $M$ is a torsion Sinnott module, 
 we obtain
\[ |\varepsilon_nM|=\ell^{f_n \sum_{h=1}^{s_n} t_{n,h}} = \ell^{f_nt_n}. \]
It is clear from the above that $r_n$, $\rho_n$ and $t_n$ are all in $\N$.
The computations of $\rank_{\Z_\ell} M(n)$, $\rank_\ell M(n)$ and 
(if $M$ is torsion) $v_\ell (|M(n)|)$ are straightforward applications 
of the decomposition \eqref{e:dfnM(n)} and the formulas in 
Section \ref{s:lOrd}.
\end{proof}

\begin{rem}\label{r:AdmModCardCharId} 
The isomorphism \eqref{EqMomega} is more convenient to assess 
cardinalities and ranks, but, by definition, the characteristic ideal 
has to be computed using \eqref{EqMomega0}. Indeed, with notations 
as above, we have 
\[ {\rm Ch}_{\L_{[\omega]}} (M_{[\omega]}) =\left\{ \begin{array}{ll}
\displaystyle{ \left(\ell_{[\omega]}^{\,t_{[\omega]}}\right)} &\ 
\text{if } r_{[\omega]}=0,\\
\ \\
(0) &\ \text{otherwise}, \end{array}\right.  \]
where we put $t_{[\omega]}:=\ds{ \sum_{j=1}^{s_{[\omega]}} t_{[\omega],j} }$.
\end{rem}

In Section \ref{s:NorSys} we shall see how normic systems provide examples
of $I_\bullet$-modules. Classical Iwasawa modules (like $\ell$-class groups and 
$\ell$-Selmer groups of abelian varieties) fit naturally in the definition of 
normic systems, hence all formulas presented here will be valid in 
relevant arithmetic settings (see Section \ref{s:ArAppl}).

\begin{rem}\label{r:RnSn}
In Iwasawa theory usually ranks and orders are computed for modules attached to subextensions $F_n$ of a pro-$p$-extension
$\calf/F$. In our case we have neat formulas for the modules arising from the sets 
$\mathscr{S}_n$, but it is the sets $\mathscr R_n$ which are associated with subextensions $F_n$ (recall that we are going to assume $\Gamma=\Gal(\calf/F)$ and 
$\Gamma_n=\Gal(F_n/F)$). For the sake of clarity,
we recall here the notations we used up to now.\smallskip

\noindent
\begin{center}\begin{tabular}{|c|c|}
\hline
{\bf Set} & {\bf Module} \\
\hline
$\mathscr{R}_n-\mathscr{R}_{n-1}$ & $e_{[n]}M=M_{[n]}$ \\
\hline
$\mathscr{R}_n$ & $\displaystyle{\sum_{m=0}^n e_{[m]}M = e_nM=M/I_nM}$ \\
\hline
$\mathscr{S}_n$ & $\varepsilon_nM$ \\
\hline
$\displaystyle{\bigcup_{m=0}^n \mathscr{S}_m}$ & 
$\displaystyle{\sum_{m=0}^n \varepsilon_mM=M(n)}$ \\
\hline
\end{tabular}\end{center}

\smallskip\noindent
When 
$\G$ is topologically finitely generated one can take 
$\G_n=\G/\G^{p^n}$. This forces
$\mathscr{S}_n=\mathscr{R}_n-\mathscr{R}_{n-1}$, so that 
$\varepsilon_nM=M_{[n]}$ and, more importantly, $M(n)=M/I_nM$. 
Hence in this setting Theorem \ref{t:AdmMod1&2}
will provide formulas for modules associated with the layers $F_n$. 
\end{rem}

\subsection{Applications}\label{SecApp1}
The results we present here derive from the main theorems of 
the previous section, hence are expressed in terms of the modules $M(n)$.
As such, they describe the growth along a tower of fields only if 
$M(n)=M/I_nM$, as discussed in Remark \ref{r:RnSn}.
  
\subsubsection{Washington's bounds}
An immediate consequence of Theorem \ref{t:AdmMod1&2} is the fact that
\[ M(n)= M(n+1) \iff  \varepsilon_{n+1}M=0 \iff 
\rank_\ell M(n) = \rank_\ell M(n+1) .\]
We obtain the following straightforward generalization of 
\cite[Proposition 1]{Wa75}.

\begin{thm}\label{t:GenWash}
A torsion Sinnott module $M$ is infinite if and only if there are infinitely 
many distinct $n\in \N$ such that $\varepsilon_nM\neq 0$.

\noindent In particular, $|M|=+\infty$ only if there are infinitely many 
distinct $n\in \N$ such that  
\begin{equation}\label{EqGenWaProp1} 
v_\ell (|M(n)|) \geqslant \rank_\ell M(n) \geqslant c\, p^n ,
\end{equation}   
where $c=c(\ell,p)$ is the constant of \eqref{e:fngnrl}.
\end{thm}
 
\begin{proof}
The first statement is obvious from $M=\prod_{n\geqslant0}\varepsilon_nM$. 
 
\noindent 
For the second one just note that, by assumption, $r_n=0$. Hence, 
by \eqref{EqMomega},
\[ v_\ell (|\varepsilon_nM|)= f_n\sum_{h=1}^{s_n} t_{n,h} 
\geqslant f_n s_n=\rank_\ell\varepsilon_nM ,\] 
because $t_{n,h}\geqslant 1$ for all $h$ (this obviously holds for 
$\varepsilon_nM=0$ as well with $s_n=0$).
Thus, for every $n\geqslant0$ such that $s_n\neq0$ 
(i.e., $\varepsilon_nM\neq 0$), one has
\[ v_\ell (|\varepsilon_nM|) \geqslant \rank_\ell \varepsilon_nM 
\geqslant f_n \geqslant
\begin{cases} f_1p^{-a(\ell,p)}\cdot p^n & 
\text{ if } p\neq 2, \vspace{4pt}\\
2^{-a(\ell,2)+1}\cdot 2^n & \text{ if }p=2 .\end{cases}  \]
where the last inequality comes from the formulas in Section \ref{s:lOrd}.
The decomposition $M(n)=\bigoplus_{m=0}^n \varepsilon_mM$ 
yields \eqref{EqGenWaProp1}.
\end{proof}

\subsubsection{Sinnott's Theorem}\label{SecSinnott}
Recently M.~Ozaki has provided a new proof and some generalizations for a 
result of W.~Sinnott (see \cite[Theorem 1]{Oz} and the references there) showing that, if $X'_p(F_n)$ is the 
prime-to-$p$ part of the class group of $F_n$ (the layers of a pro-$p$-extension $\calf/F$, with
topologically finitely generated Galois group), then the $p$-adic limit
\[ \lim_{n\rightarrow +\infty} |X'_p(F_n)| \]
is well-defined (and independent of the filtration $F_n$). 
\medskip

In our setting, the existence of such limits for Sinnott modules (for any $\ell\neq p$) is a straightforward consequence of 
Theorem \ref{t:AdmMod1&2}.

\begin{thm}\label{t:GenSinn}
Let $M$ be Sinnott module. Then, the sequences 
\[ \left\{a_n:=\rank_{\Z_\ell} M(n) \right\}_{n\in \N} 
\quad \text{and}\quad 
\left\{b_n:=\rank_\ell M(n) \right\}_{n\in \N} \]
converge $p$-adically.

\noindent If $M$ is torsion, then also the sequence
\[ \left\{ c_n:= |M(n)| = \ell^{v_\ell(|M(n)|)}\right\}_{n\in\N} \]
converges $p$-adically.
\end{thm}

\begin{proof} 
The proof of Theorem \ref{t:AdmMod1&2} yields
\[ a_{n+1}-a_n=f_{n+1}r_{n+1}\quad\text{ and }\quad b_{n+1}-b_n=
f_{n+1}\rho_{n+1} ,\] 
so that convergence is obvious from \eqref{e:fngnrl}. 
If moreover $M$ is torsion, one has 
\begin{equation} \label{e:cngrmdpn} 
c_{n+1}-c_n = c_n(\ell^{\,f_{n+1}t_{n+1}}-1) \equiv 0 \pmod{p^{n+1}} ,
\end{equation}
because $f_{n+1}$ is the order of $\ell$ in $(\Z/p^{n+1}\Z)^*$.
\end{proof} 

\begin{rem} \label{r:snntt} 
As observed in Remark \ref{r:RnSn}, usually one wants to apply the theory 
to $\Gamma=\Gal(\calf/F)$. The $p$-adic limits of Theorem \ref{t:GenSinn} 
are independent of the choice of a filtration $\calf=\bigcup_nF_n$, 
since they are defined by means of the modules $M(n)$, corresponding to 
the sets $\mathscr S_n$, which depend only on $\Gamma$. One could 
consider instead the sequences $\{a_n'\}$, $\{b_n'\}$ and $\{c_n'\}$ 
obtained replacing $M(n)$ with $e_nM$ and hence depending on the choice 
of the $F_n$'s. However, if $M$ is Sinnott then the limits of these 
new sequences are the same as the ones in Theorem \ref{t:GenSinn}, 
because finiteness of the sets $\mathscr R_n$ and $\mathscr S_n(M)$ 
implies that for each $n$ there exists $m$ such that $M(n)$ is a 
submodule of $e_mM$ and vice versa. 
\end{rem}

It is quite natural to wonder what information (if any) is encoded in 
the Sinnott limits (i.e., the limits considered in Theorem \ref{t:GenSinn}).
Most likely, one cannot obtain anything significant without adding 
some strong hypothesis. In the following, we just give a couple of 
examples where computations are easy.

\begin{ex}[Limits for free $\Lambda$-modules, 
$\Gamma\simeq\Z_p^d$] \label{ex:LimFreeMod}
We consider free $\L$-modules, with $\G$ such 
that $\mathscr{S}_n$ is finite for all $n$. Let $M=\L^r$ for some 
finite $r$. Then one has
\[ \varepsilon_n M=\bigoplus_{[\omega]\in \mathscr{S}_n} \L^r_{[\omega]} \] 
and hence
\begin{equation}\label{e:RegRank} 
\rank_{\Z_\ell} \varepsilon_n M = \rank_\ell \varepsilon_n M =
r|\mathscr{S}_n|f_n .
\end{equation}
This yields, by \eqref{e:snfn=},
\[ \rank_\bullet M(n) = \sum_{m=0}^n r |\mathscr{S}_m|f_m=
r\cdot \sum_{m=0}^n |\{\omega\in \G^\vee:\omega(\G)=\bmu_{p^m}\}| ,\] 
(with $\bullet \in\{\Z_\ell,\ell\}$).

In the particular case $\Gamma=\Z_p^d$, one has 
\[ |\{\omega\in \G^\vee:\omega(\G)=\bmu_{p^n}\}|=p^{dn}-p^{d(n-1)}=
p^{d(n-1)}(p^d-1).\]
for any $n\geqslant 1$. Since $|\mathscr{S}_0|=1$, we obtain
\begin{align*}
\rank_\bullet M(n) & = \sum_{m=0}^n r|\mathscr{S}_m| f_m=
r+r\cdot \sum_{m=1}^n  p^{d(m-1)}(p^d-1) \\
\ & = r+r(p^d-1)\sum_{m=1}^n p^{d(m-1)}, 
\end{align*}
and thus
\begin{align*} 
\lim_{n\rightarrow +\infty} \rank_\bullet M(n) & =
r+r(p^d-1)\sum_{m\geqslant 1} p^{d(m-1)} \\
\ & = r+r(p^d-1)\cdot(1-p^d)^{-1} =r-r=0.
\end{align*}
More generally, one gets the same limit of $\rank_\ell$ also for torsion 
$\L$-modules such that the second equality in \eqref{e:RegRank} holds, 
for example $M=\L/\ell\L$ (and in this case one also has
\[ \lim_{n\rightarrow+\infty}|M(n)|=1 \]
$p$-adically) or $M=\prod_{[\omega]\in \mathscr{R}} 
\L_{[\omega]}/(\ell_{[\omega]}^{n_{[\omega]}})$ with an arbitrary 
choice of $n_{[\omega]}\in \N_{>0}\cup\{+\infty\}$.
\end{ex}

\begin{ex}[Limits for stable modules] \label{ex:LimRegMod} Let $M$ be a Sinnott module and recall the numbers $r_n$, $\rho_n$ defined in \eqref{e:dfnrnrhn}. 

\noindent
We say that $M$ is {\em $\Z_\ell$-stable} (respectively, 
{\em $\ell$-stable}) if the $r_n$'s (respectively, the $\rho_n$'s) 
are constant, i.e., if either of
\begin{equation} \label{e:stabile}  
\frac{\rank_\bullet \varepsilon_n M}{f_n} = 
\frac{\rank_\bullet \varepsilon_{n+1} M}{f_{n+1}} \quad 
\text{ for all } n\geqslant 0 
\end{equation}
(with $\bullet\in\{\Z_\ell,\ell\}$) is true. Moreover, 
we say that $\rank_\bullet M$ {\em stabilizes at $m$} if the equality 
in \eqref{e:stabile} holds for $n\geqslant m$.

\noindent
Note that a free module is not stable even if $\Lambda$ is Sinnott (i.e., 
the sets $\mathscr S_n$ are all finite), because
\[ \frac{\rank_\bullet \varepsilon_n \L}{f_n} = |\mathscr{S}_n|
\quad \text{ for all } n\geqslant 0 .\]
More generally, take a locally free $\L$-module written as
\[ M =\prod_{[\omega]\in \mathscr{R}} \L_{[\omega]}^{r_{[\omega]}}
=\prod_{n\geqslant 0} \prod_{[\omega]\in \mathscr{S}_n} 
\L_{[\omega]}^{r_{[\omega]}}\simeq
\prod_{n\geqslant 0}  \Z_\ell[\bmu_{p^n}]^{r_n} . \]
Then, $M$ is stable if and only if $r_n=r_{n+1}$ for all $n\geqslant 0$.

Let $M$ be a Sinnott $\L$-module whose $\rank$ stabilizes at some 
level $N$. Without loss of generality we can assume $N\geqslant a(\ell,p)$.
Then, for $p\neq2$, Theorem \ref{t:AdmMod1&2} yields
\[ \rank_\bullet M(n)=\sum_{m=0}^{N-1} r_mf_m +\sum_{m= N}^n rf_m
= \alpha(M) +rf_1\sum_{m=N}^n p^{m-a(\ell,p)} \]
where $r$ is the ratio in \eqref{e:stabile} for $n\geqslant N$ and 
$\alpha(M)\in\N$. Taking limits on $n$ we obtain
\[ \lim_{n\rightarrow +\infty} \rank_\bullet M(n)= 
\alpha(M)+rf_1\sum_{m\geqslant N} p^{m-a(\ell,p)}
=\alpha(M)-rp^{N-a(\ell,p)}\cdot\frac{f_1}{p-1} .\]
(The case $p=2$ is similar, with trivial modifications.) 

Similarly, we say that a torsion Sinnott module is {\em $v_\ell$-stable} 
(or { $v_\ell(|M|)$ \em stabilizes at $m$}) if the ratios $t_n$ of \eqref{e:dfntn} 
are constant (or become such for $n\geqslant m$). In this case, 
a computation like the one above leads to
\[ \lim_{n\rightarrow +\infty} |M(n)|= 
\ell^{\beta(M)-tp^{N-a(\ell,p)}\cdot\frac{f_1}{p-1}} \,,\]
which is a $p$-adic number because
$\ell^{f_1}\equiv 1\pmod{p}$ and so $\sqrt[p-1]{\ell^{f_1}}\in \Z_p^*$.
\end{ex}

\section{Normic systems}\label{s:NorSys}
Recall 
\[ \Lambda = \Z_\ell [[\G]] = \varprojlim \Z_\ell\left[\G_n\right]=
\varprojlim \L_n.\]
and the exact sequence induced by the canonical projection 
\[ 1 \longrightarrow G_n \longrightarrow \G \xrightarrow{\ \pi_n\ }  \G_n \longrightarrow 1. \]

\subsection{$(\ell,p)$-normic systems} 
We adapt the definition of normic systems (originally introduced 
in \cite{vau}) to our setting. It turns out that such systems 
naturally produce $I_\bullet$-complete $\Lambda$-modules.
Moreover, the examples of class groups and Selmer groups of 
Section \ref{s:ArAppl} (and, we suspect, all arithmetically
relevant Iwasawa modules) will nicely fit the following definition.

\begin{dfn}\label{d:AdmFam}
We say that a family $\{M_n,\aa^n_m,\bb_n^m\}_{m\geqslant n\in \N}$ of 
$\L$-modules and maps 
is an {\em $(\ell,p)$-normic system} if
\begin{enumerate}[{\rm I.}]
\item $I_nM_n=0$ for all $n\in \N$, that is, $M_n$ is a $\L$-module with 
a $\L_n$-module structure induced by the canonical projection 
$\pi_n\colon\L\longrightarrow \L_n$;
\item $\{M_n,\aa^n_m\}_{m\geqslant n\in \N}$ is a direct system of 
$\L$-modules;
\item $\{M_n,\bb^m_n\}_{m\geqslant n\in \N}$ is an inverse system of 
$\L$-modules;
\item for all $n\in \N$, the composition $\bb_n^{n+1} \circ \aa^n_{n+1}$ 
is multiplication by $[G_n:G_{n+1}]$.
\end{enumerate}
\end{dfn} 

\noindent
In particular, the axioms imply that for any $\gamma\in \L$ and 
$x_n\in M_n$ one has
\[ \aa_{n+1}^n(\pi_n(\gamma)x_n)=\aa^n_{n+1} (\gamma\cdot x_n)= 
\gamma\cdot\aa^n_{n+1} (x_n)=\pi_{n+1}(\gamma)\aa_{n+1}^n(x_n). \]
 
\begin{lem}\label{l:AdmMod1}
Let $\{M_n, \aa_{n+1}^n,\bb^{n+1}_n\}_{n\in \N}$ be a normic system of 
$\L$-modules. Then, for every $n\geqslant 0$, there is a decomposition 
as direct sum of $\L$-modules
\begin{equation} \label{e:dcmpMn} 
M_{n+1} = \aa_{n+1}^n(M_n) \oplus \Ker(\bb_n^{n+1}) .
\end{equation}
\end{lem} 

\begin{proof} 
Since $\ell\neq p$, the number $[G_n:G_{n+1}]$ is invertible in $\Lambda$.
Therefore, IV implies $\aa_{n+1}^n(M_n)\cap\Ker(\bb_n^{n+1})=0$ and
\[ \bb_n^{n+1}\left(x-
\frac{\aa_{n+1}^n(\bb_n^{n+1}(x))}{[G_n:G_{n+1}]}\right)=0 \]
for every $x\in M_{n+1}$.
\end{proof} 

\begin{rem} \label{r:MapsAdmMod} 
The same reasoning also shows that the maps $\aa^n_{n+1}$ are injective, 
the maps $\bb_n^{n+1}$ are surjective, and one has 
$M_n\simeq\aa_{n+1}^n(M_n)$ for all $n$. 
\end{rem}

\subsection{Projective limits of normic systems}
In order to get uniformity in the formulas below, we 
tacitly extend a normic system 
$\{M_n,\aa^n_{n+1},\bb^{n+1}_n\}_{n\in\N}$ to $n=-1$ by putting $M_{-1}=0$.

\begin{dfn}\label{d:AdmMod}
We say that a $\L$-module $M$ is {\em normic} if there exists a normic 
system of $\L$-modules $\{M_n,\aa^n_{n+1},\bb^{n+1}_n\}_{n\in \N}$ 
such that 
\[ M\simeq\varprojlim_{n,\bb^n_{n-1}} \, M_n\,.\]
\end{dfn}

\noindent
In the setting of Definition \ref{d:AdmMod}, we shall write 
$\bb_n\colon M\rightarrow M_n$ for the natural projection.

\begin{thm}\label{t:nrmc}
Let $M$ be a normic $\L$-module, associated with the normic system 
$\{M_n,\aa_{n+1}^n,\bb^{n+1}_n\}_{n\in\N}$. Then, there is an isomorphism 
of $\Lambda$-modules
\begin{equation} \label{e:ismnrmc} 
M\simeq\prod_{n\in\N}\Ker(\bb_{n-1}^n) 
\end{equation}
such that 
\begin{equation} \label{e:ismnrmc2} 
e_{[n]}M\simeq\Ker(\bb_{n-1}^n) 
\end{equation}
for all $n\in \N$. 
In particular, every normic $\L$-module is $I_\bullet$-complete.
\end{thm}

\begin{proof}
Without loss of generality, we can assume $M=\varprojlim M_n$.

\noindent
In order to lighten notation, let $g_n:=[G_{n-1}:G_n]$. 

For each $n\in\N$, let
\[ \psi_n\colon M_n \longrightarrow \Ker(\bb^n_{n-1}) \]
be the $\Lambda$-module homomorphism
\[ \psi_n(x):=g_nx-\aa^{n-1}_n\bb^n_{n-1}(x). \]
Note that, with respect to the decomposition \eqref{e:dcmpMn}, 
$\psi_n$ vanishes on $\aa_n^{n-1}(M_{n-1})$ and induces an automorphism 
of $\Ker(\bb^n_{n-1})$.

\noindent
Define
\[ \Psi_n\colon M_n \longrightarrow \bigoplus_{k=0}^n \Ker(\bb^k_{k-1})\]
by
\[ \Psi_n(x)=\big(\psi_k\bb_k^n(x)\big)_{k=0,1,\dots,n} \,. \]
For any pair $m\geqslant n$ in $\N$, the diagram
\begin{equation} \label{e:cdpsin}
\xymatrix{ M_m \ar[rr]^{\Psi_m\qquad} \ar[d]_{\bb_n^m} & & 
\displaystyle{ \bigoplus_{k=0}^m \Ker(\bb^k_{k-1}) }\ar[d]^{\pp_n^m} \\
 M_n \ar[rr]_{\Psi_n\qquad} & & 
 \displaystyle{ \bigoplus_{k=0}^n \Ker(\bb^k_{k-1}) } }\\
\end{equation}
(where $\pp_n^m$ denotes the obvious projection) commutes, because 
of the trivial equalities $(\psi_k\bb_k^n)(\bb_n^m)(x)=\psi_k\bb_k^m(x)$.

\noindent
The maps $\Psi_n$ are isomorphisms for every $n\in\N$. This is clear 
for $\Psi_0=\psi_0$ and is proved inductively for general $n$, using 
the commutativity of \eqref{e:cdpsin} (with $m=n+1$) and the fact that 
$\bb_n^{n+1}$ induces an isomorphism of $\aa_{n+1}^n(M_n)$ into $M_n$.
Therefore, taking inverse limits on both sides of \eqref{e:cdpsin}, 
we obtain
\[ M=\varprojlim M_n\stackrel\sim\longrightarrow
\varprojlim\bigoplus_{k=0}^n \Ker(\bb^k_{k-1})=
\prod_{n\in\N}\Ker(\bb_{n-1}^n). \]
As for \eqref{e:ismnrmc2}, we start by noting that, by construction 
of the idempotents $e_n$, one has
\[ \pi_j(e_k)=1\;\;\text{ if }k\geqslant j \]
and thus, by \eqref{e:enqdr}, 
\[ \pi_j(e_{[k]})=0\;\;\text{ if }k>j. \]
The latter immediately implies 
\[ \bb_j(e_{[k]}M)=\pi_j(e_{[k]})\cdot\bb_j(M)=0\ \text{ for all }k>j,\]
and hence $\bb_n(e_{[n]}M)=\Ker(\bb_{n-1}^n)$.
\end{proof}

\noindent
The proof also shows that one has
\[ M/I_nM\simeq\prod_{j\leqslant n}\Ker(\bb_{j-1}^j) \] 
for every $n\in\N$.

\subsubsection{$I_\bullet$-modules revisited} 
Theorem \ref{t:nrmc} admits the following inverse, which we mention here 
for completeness.

\begin{thm}\label{t:nrmcInv}
Let $M$ be an $I_\bullet$-complete $\L$-module. Then, $M$ is normic.
\end{thm} 

\begin{proof}
Let $M_n=M/I_nM$. By hypothesis $M$ is (canonically) isomorphic
to the inverse limit $\liminv M_n$ defined by the
projection maps $\pi^{n+1}_n\colon M_{n+1}\longrightarrow M_n$. Hence, we
have a natural choice $\bb^{n+1}_n=\pi^{n+1}_n$ (we use $\pi^{n+1}_n$
also for the projection $\L/I_{n+1}\L\longrightarrow \L/I_n\L$,
with both interpretations appearing below).

Now, for any $x\in M_n$, take $y\in M_{n+1}$ such that
$\pi^{n+1}_n(y)=x$ and let $R_n\subset \G$ be a set of
representatives for $G_n/G_{n+1}$. We define 
$\aa_{n+1}^n\colon M_n \longrightarrow M_{n+1}$ by
\[ \aa_{n+1}^n(x):= \sum_{\tau\in R_n} \tau\cdot y.\]
Since $G_{n+1}$ acts trivially on $y\in M_{n+1}$, this definition
is independent of the choice of the set $R_n$. Moreover,
$\Ker(\bb^{n+1}_n)=\pi_{n+1}(I_n)M_{n+1}$ is generated by elements of the form $(\sigma-1)z$, with
$\sigma\in G_n$, which are killed by the operator $\sum_{R_n}\tau$, proving that $\aa_{n+1}^n(x)$ is independent also of the choice of $y$.

The maps $\bb_n^{n+1}$ are $\L$-module homomorphisms, while for the 
maps $\aa_{n+1}^n$ note that $\pi^{n+1}_n(y)=x$ yields 
$$\pi^{n+1}_n(\sigma\cdot y)=\pi^{n+1}_n(\pi_{n+1}(\sigma)y)=
\pi_n(\sigma)\pi^{n+1}_n(y)=\sigma\cdot x$$ 
for all $\sigma\in \G$. Hence
\[ \aa_{n+1}^n(\sigma\cdot x) =
\sum_{\tau\in R_n} \tau\cdot (\sigma\cdot y) =
\sum_{\tau\in R_n} (\sigma\tau)\cdot y = \sigma\cdot\aa_{n+1}^n(x).\] 
Properties I, II and III of Definition \ref{d:AdmFam} are obviously
verified. Regarding property IV, just compute
\begin{align*} (\bb_n^{n+1}\circ\aa_{n+1}^n)(x) & =
\bb_n^{n+1}\bigg(\sum_{\tau\in R_n} \tau\cdot y\bigg) 
= \sum_{\tau\in R_n} \pi^{n+1}_n(\tau\cdot y) \\
\ & = \sum_{\tau\in R_n} \pi_n(\tau)\cdot\pi^{n+1}_n(y)
=[G_n:G_{n+1}]x, 
\end{align*}
because $\tau\equiv 1\pmod{I_n}$ for all $\tau\in R_n$ (i.e.,
$\pi_n(\tau)=1$).
\end{proof}

\section{Iwasawa modules: $\ell$-class groups 
and Selmer groups}\label{s:ArAppl}

From now on, $\G$ will be the Galois group of a pro-$p$ abelian extension 
$\calf/F$ with $F$ a global field (i.e., a number field or the function 
field of a smooth projective connected curve over a finite field). 
We set $F_n=\calf^{G_n}$ as the fixed field of $G_n$, so that
\begin{equation} \label{e:calfFn} 
\calf = \bigcup_n F_n .\end{equation}
The setting is summarized in the following diagram:
\[ \xymatrix{ 
F=F_0 \ar@{-}[rrr]^{\G_n\,\simeq\, \G/G_n} \ar@/{}_{3pc}/@{-}[rrrrrrrr]_\G
\ar@/{}_{1pc}/@{-}[rrrrrr]_{\G_{n+1}\,\simeq\, \G/G_{n+1}} & & & 
F_n\ar@{-}[rrr]^{G_n/G_{n+1}} \ar@/{}^{2pc}/@{-}[rrrrr]^{G_n} & & &
F_{n+1}\ar@{-}[rr]^{G_{n+1}}  & & \calf
}\]

\begin{rem}\label{r:CharFuncFields} 
It is worth recalling that if $F$ is a global field, $\car F\neq p$ and 
$\calf/F$ is a pro-$p$-extension with finite ramification locus, then 
$\Gal(\calf/F)$ is always topologically finitely generated. 
Moreover, if $p\neq\car F> 0$, then the only available pro-$p$-extension is the
arithmetic one (see, e.g., \cite[Proposition 4.3]{BLAnnF} or 
\cite[Proposition 10.3.20]{NSW}).

\noindent
However, in the case $p=\text{char}\,F$, even bounding the ramification there 
are infinitely many $\Z_p^d$-extensions for any finite $d$, and
there are also $\Z_p^\infty$-extensions arising, for example, from the 
torsion of Drinfeld-Hayes modules (see, e.g., \cite{ABBL} and \cite{BC} 
for a study of such extensions in an Iwasawa theoretic context).
\end{rem}

\subsection{The module arising from $\ell$-class groups}\label{s:CLGR}
The notion of class group of $F$ is unambiguous if $F$ is a number field,
while if $\car F\neq0$ one can follow different conventions. 
The most classical one is to consider the group of divisor classes
(restricting to the degree $0$ subgroup if one wants to deal with a 
finite set). Alternatively, one can fix a finite set of places $S$ 
and consider the set of divisor classes for the ring of $S$-integers in 
$F$. The most interesting case, because of its role in function field
arithmetic, is probably when $S$ consists of a single place $\infty$ 
(for example, this is the case studied in \cite{ABBL}).

\noindent
In the following, by class group of $F$ we shall mean any of the finite
groups described above. Accordingly, in the function field case, the 
class group of $F_n$ will be either the class group of $S_n$-integers 
(with $S_n$ the lift of $S$ to $F_n$) or the group of degree $0$ divisor
classes.
\medskip

Let $X(F_n)=:X_n$ be the $\ell$-part of the class group of $F_n$. 
The action of $\Gamma$ on $F_n$ provides a $\Z_\ell[\G_n]$-module 
structure on $X_n$.

\noindent
Let $N^{n+1}_n$ and $i_{n+1}^n$ be the natural norm
$F_{n+1}\longrightarrow F_n$ and inclusion $F_n\longrightarrow F_{n+1}$,
respectively. We use the same notation to denote the $\L$-homomorphisms
they induce on class groups. Obviously for $x\in X_n$ one has
\[ \big(N^{n+1}_n\circ i^n_{n+1}\big)(x) = [F_{n+1}:F_n]\cdot x 
= [G_n:G_{n+1}]\cdot x.\] 
Therefore $\{X_n, i_{n+1}^n, N^{n+1}_n\}_{n\in \N}$ defines a normic 
system and we set
\[ X:=X(\calf) :=\varprojlim X_n \quad \text{with respect to norm maps}. \]
We refer to $X$ as the {\em $\ell$-class group} of $\calf$.

\noindent
The results of Sections \ref{SecSinMod} and \ref{s:NorSys} show that:
\begin{itemize}
\item $X$ is an $I_\bullet$-complete (normic) $\L$-module;
\item for every $n\in \N$,
\begin{equation}\label{e:ClGrSum}  
X/I_nX \simeq X_n \simeq \bigoplus_{m=0}^n \Ker(N_{m-1}^m) 
\simeq \bigoplus_{m=0}^n X_{[m]}
\end{equation}
and, in particular, $\Ker(N_{n-1}^n)\simeq X_{[n]}$;
\item if $\Gamma$ is topologically finitely generated, then $X$ is a 
torsion Sinnott module, satisfying Theorems \ref{t:AdmMod1&2},
\ref{t:GenWash} and \ref{t:GenSinn}   
\end{itemize}

The above results are independent of the characteristic of 
the field $F$. Note that, by Remark \ref{r:CharFuncFields}, the 
$\Lambda$-module $X$ is always Sinnott if $\car F\neq p$ and 
$\calf/F$ ramifies only at finitely many places.

\begin{rem} According to \cite{Ki}, Sinnott originally proved the 
existence of a $p$-adic limit for the sequence of ``minus'' class 
numbers in a cyclotomic $\Z_p$-extension, without restricting to 
the $\ell$-part. Similarly, the later papers \cite{Ki} and \cite{Oz} 
show the existence of a $p$-adic limit for cardinalities of the 
non-$p$-part of class groups in a $p$-adic tower. Such a result can 
be easily obtained also in our context, as follows. For each prime 
number $q$, let $X_{n,q}$ denote the $q$-part of the class group 
of $F_n$ and define $Y_n:=\bigoplus_{\ell\neq p}X_{n,\ell}$. The 
non-$p$-part of the class number is
\[ |Y_n|=\prod_{\ell\neq p}|X_{n,\ell}|, \]
where for any fixed $n$ one has $|X_{n,\ell}|=1$ for almost all $\ell$. 
By \eqref{e:cngrmdpn} we get 
\[ |X_{n+1,\ell}|\equiv|X_{n,\ell}|\pmod{p^{n+1}} \]
for every $\ell$ and $n$, and therefore 
$|Y_{n+1}|\equiv|Y_n|\pmod{p^{n+1}}$ for all $n$. 
(This is essentially the same reasoning used in the reduction of
\cite[Theorem 2]{Oz} to \cite[Proposition 1]{Oz}.)
\end{rem}

\subsubsection{A criterion for increasing the class number}
By \eqref{e:ismnrmc2}, we have $X_{[n+1]}\simeq\Ker(N^{n+1}_n)$.
Therefore, to provide a non-trivial element of $X_{[n+1]}$ we need
a non-trivial class $\bar x\in X_{n+1}$ such 
that its norm is principal. Let $x$ be an ideal or a divisor 
representing $\bar x$. Without loss of generality we can assume no 
place in the support of $x$ ramifies in $F_{n+1}/F_n$. Then, our 
assumptions imply that at least one place in the support of 
$N^{n+1}_n(x)$  splits in $F_{n+1}$ (if not, the triviality of 
$N^n_{n+1}(\bar x)$ would imply the same for $\bar x$, since $\ell$ 
does not divide $[F_{n+1}:F_n]$).

\noindent
Specializing to a $\Z_p$-extension and to class groups of rings, one 
obtains the following result.

\begin{prop} \label{p:WshinfX}
Let $F$ be a global field and $S$ a finite set of places of $F$ 
(containing the archimedean ones). Assume $\calf/F$ is a $\Z_p$-extension 
of global fields and let $X_n$ be the $\ell$-part of the class group 
of $A_n$, the ring of $S$-integers in $F_n$. Take $X=\varprojlim X_n$, 
then the module $X_{[n+1]}$ is non-trivial if and only if there exists 
a prime ideal\, $\gotp_n\subset A_0$ such that:
\begin{enumerate}[{\rm 1.}]
\item $\gotp_n$ is totally split in $F_{n+1}/F$;
\item for all $0\leqslant m\leqslant n$, all primes of $A_m$ lying 
above $\gotp_n$ are principal;
\item the primes of $A_{n+1}$ lying over $\gotp_n$ are 
not principal. 
\end{enumerate}
\end{prop}

\begin{proof} The module $X_{[n+1]}$ is non-trivial if and only if 
there is some $x$ as described above. Now note that: 
\begin{itemize}
\item in this setting, the representative $x$ can be taken to be an 
ideal $\gotq_{n+1}\subset A_{n+1}$;
\item by the Chebotarev Density Theorem, $\gotq_{n+1}$ can be assumed 
to be prime;
\item if $N_n^{n+1}(\gotq_{n+1})$ is principal, then so is 
$N_m^{n+1}(\gotq_{n+1})$ for all $m\leqslant n+1$;
\item in a $\Z_p$-extension, if a prime splits between level $n$ 
and $n+1$, it must be totally split at all levels up to $n+1$.
\end{itemize}
Thus $\gotp_n:=N_0^{n+1}(\gotq_{n+1})$ has the required properties. 
\end{proof}

\begin{rems}\label{r:ArithMod1}\  
\begin{enumerate}[{\bf 1.}] 
\item Obviously $X$ is infinite if and only if $X_{[n+1]}\neq 0$ 
for infinitely many $n$. Therefore, in the setting of Proposition
\ref{p:WshinfX}, the module $X$ is infinite if and only if there exist
infinitely many $n\geqslant 0$ for which one can find a principal prime 
$\mathfrak{p_n}$ of $A_0$ with the properties listed above.

\smallskip
\item One could compare Proposition \ref{p:WshinfX} with 
\cite[Theorem 6]{Wa75} where, for any $i\in \N-\{0\}$, the author 
builds a (non-cyclotomic) $\Z_p$-extension $\calf_i=\bigcup_n F_{i,n}$ 
of a number field $F_i$ such that $v_\ell(|X_{i,n}|)\geqslant ip^n$
for all $n\geqslant 0$ (where $X_{i,n}$ is the $\ell$-part of the
class group of $F_{i,n}$). A crucial step in the construction is 
the existence of a $\Z_p$-extension $\mathcal{K}$ of an imaginary 
quadratic field $k$ with infinitely many totally split primes. 
Roughly speaking, $F_i$ is of type $k(\bmu_\ell,\sqrt[\ell]{\alpha_i})$,
at least $i+[F_i:\Q]$ of the totally split primes (in $\mathcal{K}/k$)
are totally ramified in $F_i/k$
(the number of ramified primes obviously depends on the choice 
of $\alpha_i$), and $\calf_i=\mathcal{K}F_i$. 

\smallskip
\item In a cyclotomic $\Z_p$-extension $F_{cyc}/F$ no prime is totally 
split and \cite[Theorem]{Wa78} proves that, for an abelian number 
field $F$, the $\Lambda$-module $X(F_{cyc})$ is finite. It would
be interesting to investigate more closely the relation between the
finiteness of the $\ell$-class group and the presence of totally split
primes. We expect this to require more arithmetic information, which may be 
obtained via an approach involving also the analytic side of Iwasawa theory.

\smallskip
\item The fact that there exist cyclotomic $\Z_p$-extensions of number 
fields with infinite $p$-part of the class group shows that, in general,
there seems to be no relation between the structure of the $p$-part
and of the $\ell$-parts (compare with Section \ref{s:ArithMod2}). 
\end{enumerate}
\end{rems}

\subsection{The module arising from $\ell$-Selmer groups} \label{SecSelmer}
The situation is very similar for the other classical example of 
Iwasawa modules: the Selmer groups of an abelian variety. Let $A$ be 
an abelian variety over our global field $F$. 
For any $n\geqslant 0$, let $A[\ell^{\,n}]$ be the group scheme of 
$\ell^{\,n}$-torsion points and, for any extension $K/F$, 
let $A(K)$ (respectively, $A[\ell^{\,n}](K)$) be the group of 
$K$-rational points of $A$ (respectively, of $A[\ell^{\,n}]$). We set 
$A[\ell^{\,\infty}]:=\varinjlim A[\ell^{\,n}]$.

\noindent
For now, we assume $\car F\neq\ell$ (the other case will be briefly
discussed in subsection \ref{sss:carFell}). Consider the Galois 
cohomology groups 
\[ H^i(K,A[\ell^\infty])=
H^i\left(\Gal(K^{\text{sep}}/K),A[\ell^\infty](K^{\text{sep}})\right). \]
For any finite extension $K/F$, we define the {\em $\ell$-Selmer group} 
for $A$ over $K$ as
\begin{equation} \label{e:dfnslm} 
\Sel_\ell(A/K) : = \Ker\left( H^1(K,A[\ell^{\,\infty}])
\xrightarrow{\ \text{ res}_K\ } 
\prod_{v\in \mathscr{P}(K)} 
H^1(K_v,A[\ell^{\,\infty}])/{\rm im}(\kappa_v) \right) , 
\end{equation}
where $\mathscr{P}(K)$ is the set of all places of $K$, $K_v$ is the
completion of $K$ at $v$, $\im(\kappa_v)$ is the image of the Kummer 
map arising from the cohomology of the exact sequence of $\ell$-power 
torsion points, and the map ${\rm res}_K$ is the product of the 
canonical restrictions induced by $\text{Spec}\,K_v 
\longrightarrow \text{Spec}\,K$. The Selmer group $\Sel_\ell(A/\calf)$ 
is the direct limit of $\Sel_\ell(A/K)$ as $K$ varies among finite
subextensions of $\calf/F$.

\noindent
Recall we fixed a filtration \eqref{e:calfFn}.

\begin{lem}\label{l:SelInv}
For every $n\geqslant 0$, we have
\[ \Sel_\ell(A/F_n)\simeq \Sel_\ell(A/\calf)^{G_n}. \]
\end{lem}

\begin{proof} Consider the diagram 
\begin{equation} \label{e:dg1} 
\xymatrix{ \Sel_\ell(A/F_n) \,\ar@{^(->}[r] \ar[d] & 
H^1(F_n,A[\ell^{\,\infty}]) \ar@{->>}[r] \ar[d] & 
\im(\text{res}_{F_n}) \ar[d] \\
\Sel_\ell(A/\calf)^{G_n} \ar@{^(->}[r] & 
H^1(\calf,A[\ell^{\,\infty}])^{G_n} \ar[r]  &
\left(\im(\text{res}_\calf)\right)^{G_n} , } 
\end{equation}
where the vertical arrows are canonical restrictions. 
The Hochschild-Serre spectral sequence for the central and right 
vertical arrows, the fact that $\ell\neq p$, and the snake lemma show 
that the left vertical arrow is an isomorphism. (The reasoning is the 
same as in the proof of \cite[Theorem 4.16]{BLAnnF}, to which we 
refer for details.)
\end{proof}

In particular, for any $n\geqslant 0$, we have
\[ \Sel_\ell(A/F_n)\simeq \Sel_\ell(A/F_{n+1})^{G_n/G_{n+1}} . \]
Hence the restriction maps are injective. Composing with corestriction, 
one has the multiplication-by-$[F_{n+1}:F_n]$ map
\[ H^1(F_n,A[\ell^{\,\infty}]) \xrightarrow{\ {\rm res}_{n+1}^n\ } 
H^1(F_{n+1},A[\ell^{\,\infty}]) \xrightarrow{\ {\rm cor}_n^{n+1}\ } 
H^1(F_n,A[\ell^{\,\infty}]) ,\]
inducing the same map on the corresponding $\ell$-Selmer groups.

\noindent
Let 
\[\mathcal{S}_n:=\Sel_\ell(A/F_n)^\vee \]
be the Pontrjagin dual of $\Sel_\ell(A/F_n)$.
Dualizing, the composition
\[ \mathcal{S}_n \xrightarrow{\ ({\rm cor}^{n+1}_n)^\vee\ }
\mathcal{S}_{n+1} \xrightarrow{\ ({\rm res}_{n+1}^n )^\vee\ } 
\mathcal{S}_n \]
is also multiplication-by-$[F_{n+1}:F_n]$.
Therefore $\{ \mathcal{S}_n, ({\rm cor}^{n+1}_n)^\vee,
({\rm res}_{n+1}^n )^\vee\}_{n\in \N}$
defines a normic system of $\L$-modules and
\[ \mathcal{S}:= \varprojlim_{n,({\rm res}_{n+1}^n )^\vee} \mathcal{S}_n \]
is $I_\bullet$-complete.
 
\noindent
The results of Sections \ref{SecSinMod} and \ref{s:NorSys} show that:
\begin{itemize}
\item for every $n\in \N$,
\[ \mathcal{S}/I_n\mathcal{S} \simeq \mathcal{S}_n \simeq 
\bigoplus_{m=0}^n \Ker\! \left(({\rm res}_{m-1}^m)^\vee\right) 
\simeq \bigoplus_{m=0}^n \mathcal{S}_{[m]}  \]
and, in particular, $\Ker\!\left(({\rm res}_{n-1}^n)^\vee\right) \simeq 
\mathcal{S}_{[n]}$;
\item if $\Gamma$ is topologically finitely generated, then $\cals$ is 
a Sinnott module (because $\cals_{[\omega]}$ is a direct summand of 
$\cals_n$ when $\omega\in\Gamma_n^\vee$, and each $\cals_n$ is a 
finitely generated $\Z_\ell$-module), so that Theorems \ref{t:AdmMod1&2},
\ref{t:GenWash} and \ref{t:GenSinn} apply (with due restrictions 
if $\cals$ is not torsion).
\end{itemize}

\begin{rem}\label{RemDualSel} 
Since $\Sel_\ell(A/F_n)$ is a direct summand of $\Sel_\ell(A/F_{n+1})$, 
there is a natural isomorphism 
\[ \left(\Sel_\ell(A/F_{n+1})/\Sel_\ell(A/F_n)\right)^{\!\vee} 
\simeq \mathcal{S}_{n+1}/\mathcal{S}_n
= \Sel_\ell(A/F_{n+1})^\vee/\Sel_\ell(A/F_n)^\vee .\] 
\end{rem}

\subsubsection{The case $\car(F)=\ell$} \label{sss:carFell} 
The reasoning discussed above works in this setting as well, leading 
to the same results, after a few small changes that we quickly explain. 

Mainly, one has to replace Galois cohomology with the flat one: that 
is, given an extension $K/F$, now we denote by $H^i(K,A[\ell^\infty])$ 
the $i$-th flat cohomology group for $\text{Spec}\,K$ with 
values in the group scheme $A[\ell^\infty]$, and then proceed to define 
$\Sel_\ell(A/\calf)$ as before. This is compatible with the previous
notation, because if $\text{char}\,F\neq \ell$, then there is a 
canonical isomorphism between flat and \'etale (i.e., Galois) 
cohomology (see, e.g., \cite[III, Theorem 3.9]{Mln}).

Our Lemma \ref{l:SelInv} still holds, with the same proof, since 
diagram \eqref{e:dg1} follows from the general cohomological machinery, 
and kernels and cokernels arising from the Hochschild-Serre exact 
sequence consist of Galois cohomology groups (see 
\cite[III, Remark 2.21]{Mln}). As for the fact that 
${\rm res}_{n+1}^n\circ{\rm cores}_n^{n+1}$ is again 
multiplication-by-$[F_{n+1}:F_n]$ in this setting, we refer 
to \cite[Th\'eor\`eme 6.2.3 (Var 4) and Proposition 6.3.15]{SGA4} 
(see also \cite[Section 0.4]{Gil}).

\subsubsection{A comparison with the $p$-Selmer group}\label{s:ArithMod2} 
For any finite extension $K/F$, the $p$-Selmer group $\Sel_p(A/K)$ is 
defined as in \eqref{e:dfnslm}, replacing $A[\ell^\infty]$ with 
$A[p^\infty]$ (and using flat cohomology if $\car(F)=p$). 
For any prime $q$ there is the well-known exact sequence
involving Selmer and Tate-Shafarevich groups
\begin{equation} \label{e:sha} 
0 \rightarrow A(F_n)\otimes\Q_q/\Z_q \longrightarrow \Sel_q(A/F_n) 
\longrightarrow \sha(A/F_n)[q^\infty] \rightarrow 0.  
\end{equation}
By the Mordell-Weil theorem, the torsion of $A(F_n)$ is finite, and thus 
$\rank_\Z A(F_n)$ is the same as the $\Z_q$-corank of 
$A(F_n)\otimes\Q_q/\Z_q$. Moreover, if we assume that the groups 
$\sha(A/F_n)$ are always finite, then, considering \eqref{e:sha} for 
$q=p$ and $q=\ell$ respectively, one gets
\[ \rank_{\Z_p}\Sel_p(A/F_n)^\vee=\rank_\Z A(F_n)=
\rank_{\Z_\ell}\Sel_\ell(A/F_n)^\vee .\]
Therefore, in this setting, Theorem \ref{t:AdmMod1&2} provides estimates 
on the growth of the $\Z$-rank of the abelian variety in the tower
of the extensions $F_n$. Such estimates are not as precise as, for
example, \cite[Proposition 1.1]{MR} and \cite[Theorem 9]{TanSlmr}, 
but are completely independent from additional hypotheses on the
number or type of reduction of ramified primes.

On the other hand, \cite[Theorem 9]{TanSlmr} gives us information on
the type of $\L$-modules we might get via Selmer groups. Indeed,
assume $\Gamma\simeq\Z_p^d$, $F_n$ is the fixed field of $\Gamma^{p^n}$, 
$\calf/F$ is unramified outside a finite set of places, and $A$ has good
ordinary reduction at all such places. Under these hypotheses, 
K.-S.~Tan has proved the growth formula
\begin{equation} \label{e:tan1} 
\rank_{\Z_p}\Sel_p(A/F_n)^\vee=\kappa_1 p^{nd}+O(p^{n(d-1)}) 
\end{equation}
for some constant $\kappa_1\in\N$ (see \cite[Theorem 9]{TanSlmr} 
together with \cite[Theorem 5]{TanCntr}). Looking back at Theorem
\ref{t:AdmMod1&2} and Example \ref{ex:LimRegMod}, still
assuming the finiteness of the $\sha(A/F_n)$, we see that in the case $d=1$ 
this growth is possible only for $\Z_\ell$-stable modules.

\section{Main Conjecture: class groups of function fields}\label{s:StickEl}
In this section $F$ is a global function field of characteristic $p$ (we shall briefly deal with $\car(F)\neq p$ in Remark \ref{r:arbchrct}). The symbol $q$ will denote the cardinality of the constant field of $F$. 

As in Section \ref{s:ArAppl}, we consider an abelian pro-$p$-extension $\calf/F$ with Galois group $\G$. However, now we also fix a finite and nonempty set $S$ of places of $F$ and ask that our extension is unramified outside $S$.

For the algebraic side of our Main Conjecture, we will consider class groups, as in Section \ref{s:CLGR}. 
More precisely, to fix ideas, let $X_n$ be the $\ell$-part
of the group of divisor classes of degree $0$ of $F_n$ and let 
$X:=\displaystyle{\varprojlim X_n}$ (with respect to norms).
 
For any place $v\notin S$ we denote by $\Fr_v\in\Gamma$ the {\em geometric Frobenius} at $v$ in $\G$.

\subsection{Stickelberger series and Artin $L$-series}

\begin{dfn}\label{d:StickSer}
The {\em Stickelberger series} for $\calf/F$ is
\begin{equation}\label{eq:StSer1}
\Theta_{\calf/F,S}(u):=\prod_{v\not\in S} (1-\Fr_vu^{\deg v})^{-1} 
\in \Z[\G][[u]] .  
\end{equation}
\end{dfn}

\begin{rems} \label{Rems:StSer} \
\begin{enumerate}[{\bf 1.}]
\item The series in \eqref{eq:StSer1} is well defined, since, for any $n\in \N$,
there are only finitely many places of $F$ of degree $n$. 
\item Let $\calf_S$ be the maximal abelian extension of $F$ unramified outside $S$.
Since $\calf\subseteq \calf_S$, the series \eqref{eq:StSer1} is the 
projection to $\Z[\G][[u]]$ of the Stickelberger series 
\[ \Theta_{\calf_S/F,S}(u)\in \Z[\Gal(\calf_S/F)][[u]] \]
introduced in \cite[Section 3]{ABBL} and \cite[Section 2]{BC} in order to interpolate various $L$-functions (complex, $p$-adic and characteristic $p$).
In the current paper we just want to relate classical Artin $L$-functions to an element of our $\Lambda$, so there is no need of either fixing a place $\infty\in S$ or passing through $\calf_S$. \\
{\bf Caveat}: in comparing \eqref{eq:StSer1} with the analogous formulae in \cite{ABBL} and \cite{BC}, one should keep in mind that here we are using the geometric Frobenius, while those papers deploy the arithmetic one.
\item For any open subgroup $U$ of $\G$, let $\calf^U$ be the fixed field of $U$.
We define the Stickelberger series for $\calf^U/F$ as
\[ \Theta_{\calf^U/F,S}(u):=\prod_{v\not\in S} 
(1-\pi^\G_{\G/U} (\Fr_v)u^{\deg v})^{-1} \in \Z[\G/U][[u]] ,\]
where $\pi^\G_{\G/U}\colon \Z[\G]\longrightarrow \Z[\G/U]$ is the canonical
projection (we use the same notation for its extension to $\Z[\G][[u]]$). It is easy to see that 
\[ \Theta_{\calf^U/F,S}(u) = \pi^\G_{\G/U} (\Theta_{\calf/F,S}(u)) 
\ \ \text{ and }\ \
\Theta_{\calf/F,S}(u) = \varprojlim_U \Theta_{\calf^U/F,S}(u) .\]
\end{enumerate}
\end{rems}

\subsection{$L$-series and character evaluation}
For comparison with complex $L$-func\-tions, we fix an embedding $\ov\Q\iri\C$. Thus, we can think of characters $\chi\in\Gamma^\vee$ as taking values in $\ov\Q_\ell$ or in $\C$, according to convenience. As usual, $\chi_0$ denotes the trivial character.

Let $I_v\subset\Gamma$ denote the inertia at $v$.
For any $\chi\in \G^\vee$, we define as usual
\[ \chi(v):=\left\{ \begin{array}{ll}
\chi(\Fr_v) & \text{ if } \chi(I_v)=1,\\
\ & \\
0 & \text{ if } \chi(I_v)\neq 1, \end{array} \right. \]
where for $v\in S$ we take as $\Fr_v\in\Gamma$ any representative of the geometric Frobenius at $v$.

\begin{dfn}\label{d:ArtinLSer}
The {\em Artin $L$-function} for $\chi$ and $F$ is 
\begin{equation}\label{e:ArtinLSer1}
L(s,\chi):=\prod_{v} (1-\chi(v)q^{-s\deg v})^{-1}, 
\end{equation}
where the product is over all places $v$ of $F$ and $s$ is a complex variable. In particular, $L(s,\chi_0)=\zeta_F(s)$  is the classical zeta-function for $F$.

\noindent
The {\em Artin $S$-truncated $L$-function} is
\begin{equation}\label{e:ArtinLSer2}
L_S(s,\chi):=\prod_{v\not\in S} (1-\chi(v)q^{-s\deg v})^{-1}. 
\end{equation}
\end{dfn}

Let $\L\langle u\rangle$ be the Tate algebra of $\L[[u]]$, i.e., the subring
of power series whose coefficients tend to 0. In particular $\L\langle u\rangle$
contains all power series whose image in $\Z_\ell[\G_n][[u]]$ is a polynomial in $u$ for all $n\in \N$.

From now on, we also fix an auxiliary place  $v_0\not\in S$.
The following is the analog of \cite[Proposition 3.2]{ABBL}.

\begin{prop}\label{p:StickTateAlg}
The modified Stickelberger series
\[\Theta_{\calf/F,S,v_0}(u):= \Theta_{\calf/F,S}(u) \cdot 
(1-\Fr_{v_0} \cdot (qu)^{\deg v_0}) \]
belongs to $\L\langle u\rangle$. 
\end{prop}

\begin{proof}
Any character $\chi \in \G^\vee$ induces 
a ring homomorphism $\chi\colon\Z[\Gamma][[u]]\rightarrow\C[[u]]$, which in turn yields an equality of formal Dirichlet series
\begin{align} \label{e:interp}
\chi(\Theta_{\calf/F,S})(q^{-s}) & = \prod_{v\not\in S} (1-\chi(v)
q^{-s\deg v})^{-1} = L_S(s,\chi) \\
\ & =L(s,\chi)\prod_{v\in S} (1-\chi(v)q^{-s\deg v}). \nonumber
\end{align}

As well known, $L(s,\chi)$ converges for $\text{Re}(s)>1$; furthermore, it is a polynomial in $\Z[\chi(\G)][q^{-s}]$
for $\chi\neq \chi_0$, by \cite[Chapter VII, \S 7, Theorem 6]{Weil}.
As for $\chi_0$, by \cite[Chapter VII, \S 6, Theorem 4]{Weil}, we have
\[ L(s,\chi_0)=\frac{P(q^{-s})}{(1-q^{-s})(1-q^{1-s})} \]
with $P(u)\in \Z[u]$. Therefore, 
\begin{align*} 
\chi_0(\Theta_{\calf/F,S,v_0})(q^{-s}) & = L_S(s,\chi_0)\cdot 
(1-q^{(1-s)\deg v_0})  \\
\ & = L(s,\chi_0)\cdot (1-q^{(1-s)\deg v_0})
\prod_{v\in S} (1-q^{-s\deg v}) 
\end{align*}
is a polynomial in $\Z[q^{-s}]$ (because $S\neq \emptyset$).  

Thus, for any $n$ and  $\chi\in\Gamma_n^\vee$, we have
$$\chi(\Theta_{\calf^{G_n}/F,S,v_0})(u)=\chi(\Theta_{\calf/F,S,v_0})(u)\in\Z[\chi(\G)][u]=\Z[\chi(\Gamma_n)][u]\,.$$
This is possible only if $\Theta_{\calf^{G_n}/F,S,v_0}(u)$ belongs to $\Z[\Gamma_n][u]\subseteq\Z_\ell[\G_n][u]$ for all $n$. \normalcolor
\end{proof}

\begin{rems}\label{r:ModStickCoeff} \
\begin{enumerate}[{\bf 1.}]
\item Since $\chi(\Theta_{\calf/F,S})(q^{-s}) = L_S(s,\chi)$ 
for any $\chi$, we can also use the idempotent decomposition to define the 
Stickelberger series as the unique element of $\Z[\G][[u]]$ such that
for any $s\in \C$, $\text{Re}(s)>1$, we have
\begin{equation} \label{e:idmpdcmptht} \Theta_{\calf/F,S}(q^{-s}) = \sum_{\chi\in\Gamma^\vee} L_S(s,\chi) e_\chi .
\end{equation}
(Here $e_\chi$ is the inverse limit of the idempotents $e_{\chi,n}$, taken over all $n$ such that $\chi\in\Gamma_n^\vee$. The sum on the right-hand side of \eqref{e:idmpdcmptht} converges with respect to the inverse limit topology on $\C[[\Gamma]]=\varprojlim\C[\Gamma_n]$.) \\
For a similar approach and other recent application of Stickelberger ideals
to Main Conjectures (for $\ell=p$) see \cite[Section 3]{BP}.
\item Contrary to the original \cite[Proposition 3.2]{ABBL}, we cannot neglect
the correction term $(1-\Fr_{v_0} (qu)^{\deg v_0})$ here, because it is not
invertible in the Tate algebra, since $q^n$ does not converge to 0 in $\Z_\ell$.
\end{enumerate}
\end{rems}

\subsection{The Stickelberger element and the Main Conjecture}
Thanks to Proposition \ref{p:StickTateAlg}, we can specialize the modified Stickelberger series
at $u\mapsto 1$ (corresponding to evaluation at $s=0$ for the $L$-functions), so to obtain an element of $\Lambda$. 
By \eqref{e:struttura} and \eqref{e:idmpdcmptht}, we find that the $[\omega]$-component of $\Theta_{\calf/F,S,v_0}(1)$ is
\begin{align*} e_{[\omega]}\Theta_{\calf/F,S,v_0}(1) & =  \sum_{\chi\in [\omega]} e_\chi\Theta_{\calf/F,S,v_0}(1)  =  \sum_{\chi\in [\omega]} \chi(\Theta_{\calf/F,S,v_0})(1)e_\chi \\ 
& =  \sum_{\chi\in [\omega]} L_S(0,\chi)e_\chi. 
\end{align*}
However, as hinted in Section \ref{ss:ExtConstMC}, a different element of $\Lambda$ is more suitable for comparison with the characteristic ideal of $X$.

\begin{dfn}\label{d:ModStickEl}
The {\em modified (at $v_0$) Stickelberger element} is the unique $\theta_{\calf/F,S,v_0}$ in $\Lambda= \prod_{[\omega]\in\mathscr{R}} \L_{[\omega]}$ such that
\[ e_{[\omega]}\cdot\theta_{\calf/F,S,v_0}=(\theta_{\calf/F,S,v_0})_{[\omega]}:= 
\prod_{\chi\in [\omega]} \chi(\Theta_{\calf/F,S,v_0})(1)\;\;\;\;\text{ for all }[\omega]\in\mathscr{R}. \]
\end{dfn}

Let
\[ S_\chi:=\{v\in S\,:\,\chi(v)=1\}
=\{ v\in S\,:\, v\text{ splits completely in }\calf^{\Ker \chi}/F\}\]
(note that $S_{\chi_0}=S)$, and define
\begin{equation}\label{e:Defzn} 
z_n:=\sum_{\chi\in \G_n^\vee} |S_{\chi}|, 
\end{equation}
which will represent the number of {\em exceptional zeros 
(of level $n$)} for our Stickelberger series. \medskip

A first step towards a relation between the Stickelberger element and the order of the group
of divisor classes of degree 0 is provided by the following theorem.

\begin{thm}\label{t:StickEl+ClNum}
In the above setting, for all $n\geqslant 0$, we have
\begin{equation} \label{e:clssnmb} \lim_{s\rightarrow 0}\, (1-q^{-s})^{1-z_n} \prod_{\chi\in \G_n^\vee}
\chi(\Theta_{\calf/F,S,v_0})(q^{-s}) = h(F_n)\cdot\frac{\Omega_{n,v_0}\Omega_{n,S}}{1-q}, 
\end{equation}
where $h(F_n)$ is the order of the group of classes of divisors of degree 0 of $F_n$,
\[ \Omega_{n,v_0}= \prod_{\chi\in \G_n^\vee} (1-\chi(v_0)q^{\deg v_0}) \]
and
\[ \Omega_{n,S}=\prod_{\chi\in \G_n^\vee} \left( 
 \prod_{v\in S_{\chi}} \deg v  \prod_{v\in S-S_\chi} (1-\chi(v)) \right).\]
\end{thm}

\begin{proof}
For all $n\geqslant 0$ and $\chi\in\Gamma_n^\vee$, we have, by \eqref{e:interp},
\begin{align*}
\chi(\Theta_{\calf/F,S,v_0})(q^{-s}) & = L_S(s,\chi) (1-\chi(v_0)q^{(1-s)\deg v_0}) \\
& = L(s,\chi) (1-\chi(v_0)q^{(1-s)\deg v_0}) \prod_{v\in S} (1-\chi(v)q^{-s\deg v}). 
\end{align*}
Assume $\chi\neq\chi_0$. By the definition of $S_\chi$ and the well-known fact that $L(0,\chi)\neq 0$ (see \cite[Corollary 2 to Proposition 14.9]{Ros}), 
the product above has a zero of order $|S_\chi|$ at $s=0$. Moreover
\[ \lim_{s\rightarrow 0}
\frac{\chi(\Theta_{\calf/F,S,v_0})(q^{-s})}{(1-q^{-s})^{|S_\chi|}}=  
L(0,\chi) (1-\chi(v_0)q^{\deg v_0})\cdot \prod_{v\in S_\chi} \deg v
\cdot \prod_{v\in S-S_\chi} (1-\chi(v)) .\]
For the trivial character we have
\[ \chi_0(\Theta_{\calf/F,S,v_0})(q^{-s}) = 
\zeta_F(s) (1-q^{(1-s)\deg v_0}) \prod_{v\in S} (1-q^{-s\deg v}).\]
Since $\zeta_F$ has a pole of order 1 at $s=0$, the last term has a zero of 
order $|S|-1$ at $s=0$. Moreover, 
\[ \lim_{s\rightarrow 0}\,
\frac{\chi_0(\Theta_{\calf/F,S,v_0})(q^{-s})}{(1-q^{-s})^{|S|-1}}
=h(F)\cdot \frac{1-q^{\deg v_0}}{1-q} \prod_{v\in S} \deg v ,\]
where $h(F)$ is the order of the group of divisor classes of degree 0 of $F$
(see, e.g., \cite[Theorem 5.9]{Ros}).

Finally, recall that the characters of $\G_n^\vee$ correspond to the field $F_n$. Hence one has
\[ h(F)\cdot\prod_{\chi\in \G_n^\vee-\{\chi_0\}} L(0,\chi) = h(F_n) \]
by  \cite[Corollary 1 to Proposition 14.9]{Ros}. 

Putting all together we obtain \eqref{e:clssnmb}.
\end{proof}
 
We are only interested in the $\ell$-part, hence only in 
the $\ell$-adic valuation of the right-hand side of \eqref{e:clssnmb}. We mention a few additional hypotheses which reduce it to the $\ell$-adic valuation
of $h(F_n)$, and will enable us to provide examples for a Main Conjecture. 
\begin{itemize}
\item[{\bf H1}] Assume there is a place of $F$ which is totally split in $\calf/F$ and
take it as our $v_0$. Then, $\chi(v_0)=1$ for all $\chi$. If, moreover, $1-q^{\deg v_0}\not\equiv 0 \pmod{\ell}$, then 
\[ v_\ell\left( \frac{\Omega_{n,v_0}}{1-q}\right) = 0 \,.\]
\item[{\bf H2}] Assume that $S$ consists of all ramified places and that all of them are 
totally ramified. Then, for all $\chi\neq \chi_0$, one has $S_\chi=\emptyset$
and $\chi(v)=0$ for all $v\in S$, i.e.,
\[ \prod_{v\in S_\chi} \deg v\cdot \prod_{v\in S-S_\chi} (1-\chi(v)) = 1\,.\]
Besides, assume  that $\ell \nmid \deg v$ for all $v\in S$, then
\[ v_\ell\left( \prod_{v\in S} \deg v\right) =0 \,.\]
When all these conditions are verified one has $v_\ell(\Omega_{n,S})=0$.
\end{itemize}

\subsection{A weak IMC for the $\gotp$-cyclotomic extension}
Our goal is to formulate an Iwasawa Main Conjecture (IMC) 
relating the characteristic ideal of $X$ to the Stickelberger element
of Definition \ref{d:ModStickEl}. 
We start by providing some evidence for a weaker version of such conjecture 
in some special extensions. 

Our first set of examples comes from the theory of Drinfeld modules (see \cite{hay} for a reference covering all we use). We fix a place
$\infty$ and let $A\subset F$ be the ring of functions regular away from $\infty$. We also fix a sign function $sgn$. This determines a finite abelian extension $H_A^+/F$, where the only ramified place is $\infty$ (see \cite[\S14]{hay}). Finally we choose a $sgn$-normalized rank $1$ 
$A$-Drinfeld module (i.e., a Drinfeld-Hayes module) $\Phi$ and a second place corresponding to a prime ideal $\gotp\subset A$. Let $H_n$ be the extension obtained 
by adding to $H_A^+$ the $\gotp^n$-torsion of $\Phi$, 
and put $\calf_\gotp=\bigcup_{n\in\N}H_n$. Then $\calf_\gotp/F$ is an 
abelian Galois extension unramified outside $\{\gotp,\infty\}$, 
where the decomposition and inertia groups at $\infty$ coincide 
and have finite order $q^{\deg\infty}-1$. Moreover one has an 
isomorphism $\Gal(\calf_\gotp/H_A^+)\simeq\varprojlim (A/\gotp^n)^*$,
 where $\Gal(H_1/H_A^+)\simeq(A/\gotp)^*$ and $\Gal(\calf_\gotp/H_1)$ 
 is isomorphic to $\Z_p^\infty$. (All of this can be found in or 
 deduced from \cite[\S16]{hay}.) 

\begin{prop}\label{p:CycExt}
Let $\calf_\gotp^{(p)}/F$ be the maximal pro-$p$-subextension of $\calf_\gotp/F$.
Assume that $\calf/F$ satisfies
\begin{enumerate}[{\rm 1.}]
\item $\calf\subset \calf_\gotp^{(p)}$;
\item $p$ does not divide the class number of $A$;
\item $q^{\deg\infty}\not\equiv 1 \pmod{\ell}$;
\item $\ell \nmid \deg \mathfrak{p}$.
 \end{enumerate}
Then, one has an equality (as ideals in $\Z_\ell$)
\[ \left(\prod_{\chi\in \G_n^\vee} 
\chi(\Theta_{\calf/F,\mathfrak{p},\infty})(1)\right) 
= \big(h(F_n)\big)=(|X_n|). \]
\end{prop}

\begin{proof}
This is a consequence of Theorem \ref{t:StickEl+ClNum}. By \cite[Theorem 14.1]{hay}, assumption 2 implies that $p$ is coprime to $[H_A^+:F]$ and therefore $\Gal(\calf_\gotp/H_1)\simeq\Gal(\calf_\gotp^{(p)}/F)$. It follows that $\gotp$ is the only ramified place in $\calf/F$, so we put $S=\{\gotp\}$. Together with assumption 4, this implies {\bf H2}. Also, the place $\infty$ is totally split in $\calf/F$, because $p$ is coprime with $q^{\deg\infty}-1$. This and assumption 3 imply {\bf H1}. Moreover, $z_n=1$ for all $n$.
\end{proof} 

In particular, recalling Definition \ref{d:ModStickEl}, we have

\begin{thm}[Weak IMC for 
$\calf\subset\calf_\gotp^{(p)}$]\label{t:WMC}
In the setting of Proposition \ref{p:CycExt}, the equalities (as ideals in $\Z_\ell$)
\[\prod_{[\omega]\in \mathscr{R}_n} {\rm Ch}_{\L_{[\omega]}}(X_{[\omega]}) =
\left(\prod_{[\omega]\in \mathscr{R}_n}
(\theta_{\calf/F,\mathfrak{p},\infty})_{[\omega]}\right) \]
and
\begin{equation}\label{e:WMC}
\prod_{[\omega]\in \mathscr{R}_n-\mathscr{R}_{n-1}} 
{\rm Ch}_{\L_{[\omega]}}(X_{[\omega]}) = 
\left(\prod_{[\omega]\in \mathscr{R}_n-\mathscr{R}_{n-1}}
(\theta_{\calf/F,\mathfrak{p},\infty})_{[\omega]}\right) 
\end{equation}
hold for all $n\geqslant 0$ (with $\mathscr{R}_{-1}=\emptyset$).
\end{thm}

\begin{proof}
The first statement follows from Proposition \ref{p:CycExt}, noting that, by definition, for locally torsion $\L$-modules one has
\[\prod_{[\omega]\in \mathscr{R}_n} {\rm Ch}_{\L_{[\omega]}}(X_{[\omega]}) =
\prod_{[\omega]\in \mathscr{R}_n} (|X_{[\omega]}|)=(|X_n|).\]
Equality \eqref{e:WMC} is an immediate consequence.
\end{proof} 

\begin{ex} In the case $F=\F(t)$ (with $\F$ a finite field), one can choose $\infty$ as the place corresponding to the uniformizer $1/t$ and $sgn$ so that $sgn(1/t)=1$. It follows $A=\F[t]$ and $H_A^+=F$. In this setting, the only $sgn$-normalized rank $1$ Drinfeld module is the Carlitz module. The $\gotp$-cyclotomic Carlitz extension $\calf_\gotp/\F(t)$ is a close analogue of the $\Z_p$-cyclotomic extension of $\Q$ (for details see \cite[Chapter 12]{Ros}). The Main Conjecture for the $p$-part of the group of degree $0$ divisor classes was proved in \cite{ABBL} (see also \cite{BC} and \cite{BP} for generalizations) using (essentially) the same Stickelberger
series to provide the analytic counterpart of the characteristic ideal.
\end{ex}

\subsection{IMC for $\Z_p$-extensions} \label{ss:IMCZp}
We go back to the setting of Theorem \ref{t:StickEl+ClNum}.
In the particular case $\G\simeq\Z_p$ (with the filtration yielding $\Gamma_n\simeq\Z/p^n\Z$) we can be more precise with
the computation of $\Omega_{n,v_0}$ and $\Omega_{n,S}$. 

Let $n_0\in \N\cup\{+\infty\}$ be such that $v_0$ is totally split in 
$F_{n_0}/F$ (where $F_{+\infty}=\calf$) 
and inert in $\calf/F_{n_0}$. We use similar notations for
places $v\in S$, with $n_v^s,n_v^i\in \N\cup\{+\infty\}$ such that
$v$ is totally split in $F_{n_v^s}/F$, inert in $F_{n_v^i}/F_{n_v^s}$
and totally ramified in $\calf/F_{n_v^i}$.

For any $n$ and any $\sigma\in\Gamma_n$, let $ev_\sigma\colon\Gamma_n^\vee\rightarrow\boldsymbol\mu$ be the evaluation map $\chi\mapsto\chi(\sigma)$. There is an obvious exact sequence
\begin{equation} \label{e:exseqev} 1 \longrightarrow \Ker(ev_\sigma) \longrightarrow \Gamma_n^\vee  \longrightarrow \im(ev_\sigma)  \longrightarrow 1. \end{equation}

\subsubsection{The factor $\Omega_{n,v_0}$} Let $\sigma\in\Gamma_n$ be the image of $\Fr_{v_0}$. By the identity
$$\prod_{\zeta\in\im(ev_\sigma)}(1-\zeta x)=1-x^{|\im(ev_\sigma)|}$$
and \eqref{e:exseqev}, it follows
\begin{align*} \Omega_{n,v_0} & =\prod_{\chi\in \G_n^\vee} (1-\chi(v_0)q^{\deg v_0})\;\;=\;\;\left(\prod_{\zeta\in\im(ev_\sigma)}(1-\zeta q^{\deg v_0})\right)^{|\Ker(ev_\sigma)|}\\
 & =\big(1-q^{|\im(ev_\sigma)|\deg v_0}\big)^{|\Ker(ev_\sigma)|}
\end{align*}
Note that $\Ker(ev_\sigma)=\Gamma_n^\vee$ if $n\leqslant n_0$ and $|\Ker(ev_\sigma)|=p^{n_0}$ if $n>n_0$. Hence
\begin{equation} \label{e:vltOmg0} \Omega_{n,v_0} =
 \left\{\begin{array}{ll} \left(1-q^{\deg v_0}\right)^{p^n} & \text{ if } n\leqslant n_0,\\
 & \\
\left(1-q^{p^{n-n_0}\deg v_0}\right)^{p^{n_0}} &\text{ if } n> n_0\,. \end{array}\right.\end{equation}

\subsubsection{The factor $\Omega_{n,S}$} \label{sss:vltOmgS}
For $\chi\in\Gamma_n^\vee$ and $v\in S$, let
\[ a_{\chi,v}:=\left\{\begin{array}{ll} \deg v & \text{ if } v\in S_\chi,\\
 & \\
1-\chi(v) & \text{ if } v\notin S_\chi, \end{array}\right.\]
so to have
\[\Omega_{n,S}=\prod_{\chi\in \G_n^\vee} \left( 
 \prod_{v\in S_{\chi}} \deg v  \prod_{v\in S-S_\chi} (1-\chi(v)) \right)=\prod_{v\in S}\prod_{\chi\in \G_n^\vee}a_{\chi,v}. \]
Recall that $v\in S$ is split in $F_{n_v^s}/F$, inert in the $F_{n_v^i}/F_{n_v^s}$ and totally ramified above $F_{n_v^i}$. Hence, the $v$-factor $\prod_{\chi\in \G_n^\vee}a_{\chi,v}$ in $\Omega_{n,S}$ is 
$$(\deg v)^{p^n} \quad\quad \text{ for }n\leqslant n_v^s.$$ 

For $n_v^s<n\leqslant n_v^i$, let $\sigma\in\Gamma_n$ be the geometric Frobenius at $v$ and note that we have $\chi(v)=\chi(\sigma)$ for all $\chi\notin\Ker(ev_\sigma)$. Hence the equality
\[ \prod_{\chi\notin\Ker(ev_\sigma)}(1-\chi(\sigma)x)=\prod_{\chi\in\Gamma_n^\vee}(1-\chi(\sigma)x)\;\cdot\prod_{\chi\in\Ker(ev_\sigma)}(1-x)^{-1}=\left(\frac{1-x^{|\im(ev_\sigma)|}}{1-x}\right)^{|\Ker(ev_\sigma)|} \]
implies 
\[ \prod_{\chi\notin\Ker(ev_\sigma)}(1-\chi(v))=|\im(ev_\sigma)|^{|\Ker(ev_\sigma)|}. \]
Besides, $|\Ker(ev_\sigma)|=p^{n_v^s}$ and so $|\im(ev_\sigma)|=p^{n-n_v^s}$.
Therefore the $v$-factor in $\Omega_{n,S}$ is
\[ (\deg v)^{p^{n_v^s}}\cdot p^{(n-n_v^s)p^{n_v^s}} \quad\quad \text{ for }n_v^s<n\leqslant n_v^i .\]

Finally, for $n>n_v^i$ and $\chi\in\Gamma_n^\vee-\Gamma_{n_v^i}^\vee$, we have $a_{\chi,v}=1$. Hence the $v$-factor in $\Omega_{n,S}$ is
\[ (\deg v)^{p^{n_v^s}}\cdot p^{(n_v^i-n_v^s)p^{n_v^s}} 
\quad\quad \text{ for }n>n_v^i.\]

\subsubsection{Iwasawa Main Conjecture}
We keep $\calf/F$ a $\Z_p$-extension, $v_0$ and $S$ as above.

\begin{lem} \label{l:vomgbndd}
Assume that
\begin{enumerate}[{\rm 1.}]
\item $v_0$ does not split completely in $\calf/F$ or $v_\ell(1-q^{\deg v_0})=0$;
\item if $v\in S$, then $v$ does not split completely in $\calf/F$ 
%$n_v^s<+\infty$ 
or $\ell$ is coprime with $\deg v$.
\end{enumerate}
Then, $v_\ell(\Omega_{n,v_0}\Omega_{n,S})$ stabilizes at some (effectively computable) $\overline n$.
\end{lem}

\begin{proof}
If $n_0<+\infty$ and there exists $r\geqslant0$ such that $q^{p^r\deg v_0}\equiv1\pmod\ell$ it follows
\[  v_\ell(1-q^{p^r\deg v_0}) = v_\ell(1-q^{p^{r+m}\deg v_0}) 
\quad \text{for all }m\geqslant 0. \]
Thus, assumption 1 and \eqref{e:vltOmg0} imply that the $\ell$-adic 
valuation of $\Omega_{n,v_0}$ stabilizes at some $n=m_0$ 
which is effectively computable once $q$, $\ell$ and $v_0$ are given
(with $m_0=0$ when $n_0=+\infty$). 
As for $\Omega_{n,S}$, the computations of Section \ref{sss:vltOmgS} immediately show that each $v$-factor has bounded $\ell$-adic valuation under assumption 2, reaching the maximum at
\[ m_S:=\max\big\{n_v^s:v\in S,\ \deg v\in\ell\,\Z\big\} \]
(with $m_S:=0$ if $\ell$ never divides $\deg v$). To conclude, take $\overline n:=\max\{m_0,m_S\}$.
\end{proof}

\begin{thm}\label{t:MCArith}
Let $\calf/F$ be a $\Z_p$-extension.
Fix a nonempty finite $S$ of places of $F$ containing all the ramified ones and a place $v_0\notin S$.
Assume that\begin{enumerate}[{\rm 1.}]
\item $v_0$ does not split completely in $\calf/F$ or $v_\ell(1-q^{\deg v_0})=0$;
\item no place in $S$ splits completely;
\item $\ell$ is a generator of $(\Z/p\Z)^*$ and $v_p(\ell^{p-1}-1)=1$.
\end{enumerate}
Then, there exists an effectively computable $\overline n\in\N$ such  
\[ {\rm Ch}_{\L_{[\omega]}}(X_{[\omega]})= \left((\theta_{\calf/F,S,v_0})_{[\omega]}\right). \] 
for all $[\omega]\in\mathscr{R}-\mathscr{R}_{\overline n}$.
\end{thm} 

\begin{proof}
By Theorem \ref{t:StickEl+ClNum} we have an equality of ideals in $\Z_\ell$
\begin{equation} \label{e:rcllthmlmt} 
\left( \lim_{s\rightarrow 0}\, (1-q^{-s})^{1-z_n} 
\prod_{\chi\in \G_n^\vee} \chi(\Theta_{\calf/F,S,v_0})(q^{-s})\right) 
= \left(h(F_n)\cdot\frac{\Omega_{n,v_0}\Omega_{n,S}}{1-q}\right).  \end{equation}
The number of exceptional zeroes stabilizes at $\max_{v\in S}\{n_v^s\}$ (which is finite by assumption 2).
Moreover, by Lemma \ref{l:vomgbndd}, the $\ell$-adic valuation of 
$\frac{\Omega_{n,v_0}\Omega_{n,S}}{(1-q)}$ stabilizes at some (computable) 
$\overline{m}$. Hence, for all 
$n\geqslant \overline{n}:=\max\{ \overline{m},\,n_v^s\,:\,v\in S\}$ we obtain
\begin{align*}
\left(\frac{\Omega_{n,v_0}\Omega_{n,S}}{\Omega_{n-1,v_0}\Omega_{n-1,S}}\cdot \frac{h(F_n)}{h(F_{n-1})}\right) & =\left(\frac{h(F_n)}{h(F_{n-1})}\right)=
\left( \lim_{s\rightarrow 0}\, \prod_{\chi\in \G_n^\vee-\G_{n-1}^\vee} \chi(\Theta_{\calf/F,S,v_0})(q^{-s})\right) \\
 \ & = \left( \prod_{\chi\in \G_n^\vee-\G_{n-1}^\vee} \chi(\Theta_{\calf/F,S,v_0})(1)\right) \\
 \ & = \left( \prod_{[\omega]\in \mathscr{R}_n-\mathscr{R}_{n-1}} (\theta_{\calf/F,S,v_0})_{[\omega]})\right).
\end{align*}
As noted at the very end of Section \ref{s:lOrd}, the hypothesis on $\ell$ implies
$f_n(\ell,p)=\varphi(p^n)$ for all $n$. Therefore, since $\Gamma\simeq\Z_p$, each $\mathscr S_n=\mathscr R_n-\mathscr R_{n-1}$ consists of a single class. Let $[\omega]\in\mathscr{S}_n$, with $n\geqslant\overline n$. The theorem follows from
\[ {\rm Ch}_{\L_{[\omega]}}(X_{[\omega]}) = 
(|X_{[n]}|)=\left(\frac{|X_n|}{|X_{n-1}|}\right)
=\left(\frac{h(F_n)}{h(F_{n-1})}\right), \]
where the middle equality is justified by \eqref{e:ClGrSum}.
\end{proof}

\subsubsection{Special cases}
\begin{enumerate}[1.]
\item For a first example, we go back to the special case of the $\gotp$-cyclotomic extension. Recall that, assuming that
$p$ does not divide the class number of $A$, 
the unique ramified place in $\calf_\gotp^{(p)}/F$ is $\gotp$, which is totally ramified, while $\infty$ is totally split. 

\begin{cor}{\em \textbf{[IMC for $\Z_p$-subextensions of 
$\calf_\gotp$]}} \label{co:IMCZp} 
Let $\calf/F$ be a $\Z_p$-extension contained in $\calf_\gotp^{(p)}$.
Fix $S=\{\gotp\}$ and $v_0=\infty$. 
Assume that $p$ does not divide the class number of $A$,
$v_\ell(1-q^{\deg\infty})=0$, $\ell$ is a generator of 
$(\Z/p\Z)^*$ and $v_p(\ell^{p-1}-1)=1$.

\noindent
Then, for all $[\omega]\in \mathscr{R}-\mathscr{R}_{\overline n}$ we have
\[ {\rm Ch}_{\L_{[\omega]}}(X_{[\omega]})= \left((\theta_{\calf/F,\gotp,\infty})_{[\omega]}\right). \]
\end{cor}

\noindent
The condition $v_\ell(1-q^{\deg\infty})=0$ can be dropped by taking as $v_0$ any unramified place which does not split completely in $\calf/F$ (if such places exist).
\item Another classical example of $\Z_p$-extension for a global function field $F$ 
is the {\em arithmetic} extension $F^{ar}/F$ obtained by composing with 
the $\Z_p$-extension of the constant field $\mathbb{F}_F$. 
Such extension is disjoint from the $\gotp$-cyclotomic extension 
which is {\em geometric} (i.e., $F$ and $\calf_\gotp$ have the 
same field of constants).

\noindent
The arithmetic extension is unramified at all places. Since we required a nonempty set $S$, we choose arbitrary $v_0$ and $v_1$, putting
$S=\{v_1\}$. We also recall that no place is totally split in $F^{ar}/F$.

\begin{cor}{\em \textbf{[IMC for the arithmetic $\Z_p$-extension]}} \label{co:IMCZpArith} 
Let $F^{ar}/F$ be the arithmetic $\Z_p$-extension,
and fix $v_0$ and $S=\{v_1\}$ as above.
Assume $\ell$ is a generator of $(\Z/p\Z)^*$ and $v_p(\ell^{p-1}-1)=1$.

\noindent
Then, for all $[\omega]\in\mathscr R-\mathscr{R}_{\overline n}$ we have
\[ {\rm Ch}_{\L_{[\omega]}}(X_{[\omega]})= \left((\theta_{F^{ar}/F,v_1,v_0})_{[\omega]}\right)\,.\]
\end{cor}

\noindent
It is well known (see, e.g., \cite[Section II]{Wa75}) that the power
of $\ell$ dividing $h(F_n)$ is bounded in the arithmetic $\Z_p$-extension
$F^{ar}/F$. Therefore for all $n\gg 0$  and $\mathscr{R}_n=\{[\omega]\}$
we have $X_{[\omega]}=0$ and $(\theta_{F^{ar}/F,v_1,v_0})_{[\omega]}$ 
is a $\ell$-adic unit.  
\end{enumerate}

\begin{rem} \label{r:arbchrct}
In our discussion up to here we assumed that $F$ has characteristic $p$ only to have plenty of $\Z_p$-extensions. Actually, our construction of the Stickelberger element and Corollary \ref{co:IMCZpArith} work with no change in any positive characteristic.

The growth of class groups in arithmetic extension of function fields was studied in detail by Rosen (\cite{Ros1}; see also \cite[Chapter 11]{Ros}).
\end{rem}

\subsection{An Iwasawa Main Conjecture for class groups?} \label{sss:2stck} 
We go back to the setting of $\calf/F$ an abelian pro-$p$-extension with finite ramification contained in $S$. In light of the previous results, it seems reasonable to formulate an Iwasawa Main Conjecture of the following type:\medskip

\begin{itemize}
\item[{\bf [IMC]}] {\em there exists $\overline{n} \in \N$ such that for any $[\omega]\in \mathscr{R}-\mathscr{R}_{\overline{n}}$, we have
\begin{equation}\label{e:IMC} {\rm Ch}_{\L_{[\omega]}}(X_{[\omega]})=
\big((\theta_{\calf/F,S,v_0})_{[\omega]}\big) 
\end{equation}
as ideals in $\L_{[\omega]}$. } \end{itemize}
\medskip

\noindent We also expect similar relations for other Iwasawa modules once one is
able to provide the ``right'' $L$-functions for them.

\begin{rem} The computations of Section \ref{ss:IMCZp} show that one has to expect some irregularities for small $n$. This should be compared with the error term $\nu\in\Z$ in Iwasawa's formula \eqref{e:iwswfrm}. Actually, the main reason for the appearance of $\nu$ in \eqref{e:iwswfrm} is the fact that the structure of finitely generated torsion $\Z_p[[\Z_p]]$-modules is determined only up to pseudo-isomorphism: that is, up to pseudo-null kernels and cokernels (recall that a module is pseudo-null if its annihilator has height at least $2$).
Characteristic ideals are defined up to pseudo-isomorphisms:
hence pseudo-null modules do not appear in the Main Conjecture 
but they contribute to formula \eqref{e:iwswfrm}.

In our situation, pseudo-isomorphisms play no role: the ring $\Lambda$ has Krull dimension $1$ and thus the only pseudo-null module is the trivial one. Moreover, the discussion in Sections \ref{s:CharId} and \ref{SecSinMod} makes it clear that two $I_\bullet$-complete $\Lambda$-modules $M$ and $N$ are isomorphic if and only if so are their $[\omega]$-components, for all $[\omega]$. The additional precision, with respect to classical Iwasawa theory, that we gain in this is however balanced by the fact that the $[\omega]$-components are (in general) independent of each other. 
 \end{rem}

We can propose also an alternative approach to {\bf [IMC]}, based on Sections \ref{ss:CharDec} and \ref{ss:ExtConstMC}.
For any $\chi\in[\omega]$, let $X_{[\omega]}^{(\chi)}$ be the $\chi$-part of $X_{[\omega]}$,
as in Proposition \ref{p:cardecM}. Extending constants we can formulate a $\chi$-version
of the Main Conjecture as
\medskip

\begin{itemize}
\item[{\bf [$\boldsymbol{\chi}$-IMC]}] {\em there exists $\overline{n} \in \N$ such that for any $[\omega]\in \mathscr{R}-\mathscr{R}_{\overline{n}}$ 
and $\chi\in [\omega]$,  we have
\begin{equation}\label{e:chiIMC} 
{\rm Ch}_{\L_{[\omega]}\otimes_{\Z_\ell} \Z_\ell[\chi(\G)]}(X_{[\omega]}^{(\chi)})
= \left(e_\chi(\Theta_{\calf/F,S,v_0})(1)\right)
\end{equation}
as ideals in $\L_{[\omega]}\otimes_{\Z_\ell} \Z_\ell[\chi(\G)]$. } \end{itemize}
\medskip

It is easy to see that \eqref{e:IMC} and \eqref{e:chiIMC} are equivalent. Indeed, the algebraic sides are linked by \eqref{e:cardecid}, which yields
$$ {\rm Ch}_{\L_{[\omega]}}(X_{[\omega]})\otimes_{\Z_\ell} \Z_\ell[\omega(\G)]= 
\left({\rm Ch}_{\L_{[\omega]}\otimes_{\Z_\ell} \Z_\ell[\omega(\G)]}(X_{[\omega]}^{(\omega)})\right)^{|[\omega]|}.$$
As for the analytic sides, let $\chi$ and $\chi'$ be two elements in the class $[\omega]$. Then one has $\chi'=\sigma\cdot\chi$ for some $\sigma\in\Gal(\overline\Q_\ell/\Q_\ell)$ and therefore
\[e_{\chi'}(\Theta_{\calf/F,S,v_0})(1)=\sigma\big(e_\chi(\Theta_{\calf/F,S,v_0})(1)\big),\]
implying
\[\big((\theta_{\calf/F,S,v_0})_{[\omega]}\big)\otimes_{\Z_\ell}\Z_\ell[\omega(\G)]=\left(e_\omega(\Theta_{\calf/F,S,v_0})(1)\right)^{|[\omega]|}\]
as ideals generated by elements with the same $\ell$-adic valuation.

The formulation {\bf [IMC]} has a more familiar look for those working in Iwasawa theory, since it 
is expressed in terms of elements in the Iwasawa algebra. However, {\bf [$\boldsymbol{\chi}$-IMC]} might be more amenable to proof.


\begin{thebibliography}{30}

\bibitem{ABBL} {\sc B. Angl\'es - A. Bandini - F. Bars - I. Longhi} 
Iwasawa main conjecture for the Carlitz cyclotomic extension and applications, 
{\em Math. Ann.} {\bf 376} (2020), 475--523.

\bibitem{BaHo} {\sc P.N. Balister - S. Howson} Note on Nakayama's lemma for compact $\Lambda$-modules, 
{\em Asian J. Math.} {\bf 1} (1997), 224--229.

\bibitem{BBLNY} {\sc A. Bandini - F. Bars - I. Longhi} 
Characteristic ideals and Iwasawa theory, 
{\em New York J. Math.} {\bf 20} (2014), 759--778.

\bibitem{BC} {\sc A. Bandini - E. Coscelli} Stickelberger series and 
Main Conjecture for function fields,
{\em Publ. Mat.} {\bf 65}  (2)  (2021), 459--498.

\bibitem{BL1} {\sc A. Bandini - I. Longhi} Control theorems for elliptic 
curves over function fields,
{\em Int. J. Number Theory} {\bf 5} (2009), no. 2, 229--256.

\bibitem{BLAnnF} {\sc A. Bandini - I. Longhi}  Selmer groups for elliptic 
curves in $\mathbb Z_l^d$-extensions of function fields of char $p$, 
{\em Ann. Inst. Fourier}, {\bf 59} (2009), no. 6, 2301--2327.

\bibitem{BP} {\sc W. Bley - C.D. Popescu} Geometric main conjectures in 
function fields, {\em J.~Reine Angew.~Math.} {\bf 818} (2025), 1--33.

\bibitem{BD} {\sc D. Burns - A. Daoud} On Iwasawa theory over $\Z[[\Z_p]]$, preprint

\bibitem{BDL} {\sc D. Burns - A. Daoud - D. Liang} On non-noetherian 
Iwasawa theory, preprint, arXiv:2401.02946 [math.NT].

\bibitem{SGA4} {\sc P. Deligne} Cohomologie a supports propres,
in {\em S\'eminaire de G\'eom\'etrie alg\'ebrique de l'I.H.E.S., SGA4
Expos\'e XVII}.

\bibitem{Gil} {\sc P. Gille} La $R$-\'equivalence sur les groupes 
alg\'ebriques r\'eductifs d\'efinis sur un corps global,
{\em Publ. Math., Inst. Hautes \'Etud. Sci.} {\bf 86} (1997), 199--235.

\bibitem{GrI} {\sc R. Greenberg} Iwasawa theory - past and present,
in {\em Class field theory - its centenary and prospect. Proceedings of the 7th MSJ 
International Research Institute of the Mathematical Society of Japan} Miyake, Katsuya (ed.), 
Adv. Stud. Pure Math. {\bf 30} (2001), 335--385. 

\bibitem{hay} {\sc D.~Hayes} A brief introduction to Drinfeld modules, in {\em The Arithmetic of Function Fields, Columbus, OH, 1991}, de Gruyter (1992), pp. 1--32.

\bibitem{Ki} {\sc H. Kisilevsky} A generalization of a result of Sinnott, 
{\em Pacific J. Math.} {\bf 181} (1997), 225--229.

\bibitem{KL} {\sc D. Kundu - A. Lei} Growth of $p$-parts of ideal class groups and fine 
Selmer groups in $\Z_q$-extensions with $p\neq q$, 
{\em Acta Arith.} {\bf 207} (2023), 297--313.

%\bibitem{lltt} {\sc K.F. Lai - I. Longhi - K.-S. Tan - F. Trihan} Pontryagin %duality for Iwasawa modules and abelian varieties. {\em Trans.~Amer.~Math.~Soc.} {\bf 370} (2018),  no.~3, 1925--1958.

\bibitem{LV} {\sc A. Lei - D. Valli\`eres} The non-$\ell$-part of the number of spanning trees 
in abelian $\ell$-towers of multigraphs, 
{\em Research in Number Theory} {\bf 9} (2023), article number 18.

\bibitem{MR} {\sc B. Mazur - K. Rubin} 
Studying the growth of Mordell-Weil,
{\em Doc. Math.} Extra Vol. Kato (2003), 585--607.

\bibitem{Mln} {\sc J.S.~Milne} \'Etale Cohomology, Princeton Mathematical Series {\bf33}, Princeton
University Press, Princeton, N. J., (1980).


\bibitem{NSW} {\sc J. Neukirch - A. Schmidt - K. Wingberg} Cohomology of number fields, Grundlehren der mathematiscen
Wissenschaften {\bf 323}, 2nd Ed. Springer-Verlag (2008).

\bibitem{Oz} {\sc M. Ozaki} On the $p$-adic limit of class numbers along a pro-$p$-extension,
arXiv:2306.08407v1 [math.NT] (2023).

\bibitem{Ros1} {\sc M. Rosen}, The asymptotic behavior of the class group of a function field over a finite field, {\em Arch. Math. (Basel)}
{\bf 24} (1973), 287--296.

\bibitem{Ros} {\sc M. Rosen} Number theory in function fields,
{\bf GTM} 210, Springer-Verlag (2002). 

\bibitem{TanCntr} {\sc K.-S.~Tan} A generalized Mazur's theorem and its applications, {\em Trans. Amer. Math. Soc.} {\bf 362}
(2010), 4433--4450.

\bibitem{TanSlmr} {\sc K.-S.~Tan}  Selmer groups over $\Z_p^d$-extensions, {\em Math. Ann.} {\bf 359} (2014), 1025--1075.


\bibitem{vau} {\sc D.~Vauclair} Sur la dualit\'e et la descente d'Iwasawa, {\em Ann. Inst. Fourier (Grenoble)}
{\bf 59} (2009), no.~2, 691--767. 

\bibitem{Wa75} {\sc L.C. Washington} Class numbers and $\Z_p$-extensions,
{\em Math. Ann.} {\bf 214} (1975), 177--193.

\bibitem{Wa78} {\sc L.C. Washington} The non-$p$-part of the class number in a cyclotomic $\Z_p$-extension,
{\em Inventiones Math.} {\bf 49} (1978), 87--97.

\bibitem{Weil} {\sc A. Weil} Basic number theory, Classics in Mathematics, 
Springer-Verlag (1995).



\end{thebibliography}
\end{document}